\documentclass[a4paper, 10pt, twoside, notitlepage]{amsart}

\usepackage{amsmath,amscd}
\usepackage{amssymb}
\usepackage{amsthm}
\usepackage{comment}
\usepackage{graphicx, xcolor}

\usepackage{mathrsfs}
\usepackage[ocgcolorlinks, linkcolor=blue]{hyperref}

\usepackage{bm}
\usepackage{bbm}
\usepackage{url}

\usepackage[utf8]{inputenc}
\usepackage{mathtools,amssymb}
\usepackage{esint}
\usepackage{tikz}
\usepackage{dsfont}
\usepackage{relsize}
\usepackage{url}
\usepackage{xcolor}
\usepackage{graphicx}
\usepackage{mathrsfs}
\usepackage[shortlabels]{enumitem}
\usepackage{lineno}
\usepackage{amsmath}
\usepackage{enumitem}
\usepackage{amsthm} 
\usepackage{verbatim}
\usepackage{dsfont}
\numberwithin{equation}{section}

\usepackage{esint}

\allowdisplaybreaks

\mathtoolsset{showonlyrefs}

\graphicspath{{images/}}

\newtheorem{theorem}{Theorem}[section]
\newtheorem{lemma}[theorem]{Lemma}

\newtheorem{corollary}[theorem]{Corollary}

\title[Fractional Calderón problem with $L^{\infty}$-conductivities]{Fractional Calderón problem with $L^{\infty}$-conductivities}

\author[G. Uhlmann]{Gunther Uhlmann}
\address{Department of Mathematics, University of Washington, Seattle, WA 98195-4350, USA}
\email{gunther@math.washington.edu}

\author[P. Zimmermann]{Philipp Zimmermann}
\address{Institute of Mathematics, École Polytechnique Fédérale de Lausanne (EPFL), Lausanne, Switzerland}
\email{philipp.zimmermann@epfl.ch}

\date{\today}


\newcommand{\R}{{\mathbb R}}

\newcommand{\N}{{\mathbb N}}

\newcommand{\eps}{\varepsilon}

\newcommand{\schwartz}{\mathscr{S}}

\newcommand{\tempered}{\mathscr{S}^{\prime}}

\newcommand{\fourier}{\mathcal{F}}
\newcommand{\ifourier}{\mathcal{F}^{-1}}

\newcommand{\distr}{\mathscr{D}^{\prime}}

\DeclareMathOperator{\supp}{supp} 
\DeclareMathOperator{\dist}{dist} 
\DeclareMathOperator{\loc}{loc} 

\begin{document}

	\maketitle
	\begin{abstract}
     We study the inverse problem for the fractional conductivity
equation with uniformly elliptic conductivities
$\gamma\in L^\infty(\mathbb R^n)$. For every $n\geq1$ and
$0<s<1$, we prove that the local exterior Dirichlet-to-Neumann
map on any nonempty open subset of the exterior of a bounded
open domain determines $\gamma$ uniquely almost everywhere
in $\mathbb R^n$. No continuity, Sobolev regularity, or
prescribed behavior at infinity is required.

The proof works directly with the conductivity bilinear form
and avoids the fractional Liouville reduction. It combines
exterior determination at Lebesgue points, equality of weighted
solutions, and a unique continuation principle for the
fractional Laplacian in the weighted class
\[
    L_s=\left\{
        g\in L^1_{\mathrm{loc}}(\mathbb R^n):
        \int_{\mathbb R^n}
        \frac{|g(x)|}{1+|x|^{n+2s}}\,dx<\infty
    \right\}.
\]
Specifically, if $g\in L_s$ and both $g$ and $(-\Delta)^s g$
vanish in a nonempty open set, with the latter equality
understood distributionally, then $g=0$ almost everywhere.
This result requires neither local boundedness nor Sobolev
regularity of $g$. Its proof uses the Caffarelli--Silvestre
extension and boundary unique continuation for the associated
degenerate elliptic equation.
	\end{abstract}

	\tableofcontents
	
\section{Introduction}
\label{sec: introduction}

The classical Calderón problem \cite{calderon2006inverse} asks whether an unknown conductivity $\gamma$ can be recovered from boundary measurements of voltage and current. In its nonlocal analogues, one replaces the \emph{conductivity equation}
\begin{equation}
\label{eq: conductivity equation}
\begin{cases}
     \operatorname{div}\,(\gamma\nabla u)=0\quad&\text{in }\Omega,\\
     u=\varphi,\quad &\text{on }\partial\Omega
\end{cases}
\end{equation}
and the \emph{Dirichlet to Neumann (DN) map} 
\begin{equation}
    \varphi\mapsto \Lambda_\gamma \varphi\vcentcolon = \gamma \partial_\nu u_{\varphi},
\end{equation}
which encodes the voltage/current measurements, by corresponding nonlocal quantities. This line of research has been initiated by the seminal work \cite{GSU20}, which studies the Calderón problem for the \emph{fractional Schrödinger equation}
\begin{equation}
\label{eq:fractional-schroedinger}
\begin{cases}
     ((-\Delta)^s+q)v=0\quad&\text{in }\Omega,\\
     v=\psi,\quad &\text{in }\Omega_e,
\end{cases}
\end{equation}
where $\Omega_e= \R^n\setminus \overline{\Omega}$ is the exterior of $\Omega$, $q\in L^\infty (\R^n)$ is a nonnegative potential, $0<s<1$, and $(-\Delta)^s$ is the \emph{fractional Laplacian}:
\begin{equation}
\label{eq: frac Lap}
    (-\Delta)^s u=\mathcal{F}^{-1}(|\xi|^{2s}\hat{u}).
\end{equation}
In the classical limit $s=1$, equation \eqref{eq:fractional-schroedinger} reduces to the Schrödinger equation
\begin{equation}
\label{eq:schroedinger-equation}
    \begin{cases}
     (-\Delta +q)v=0\quad&\text{in }\Omega,\\
     v=\psi,\quad &\text{on }\partial\Omega,
\end{cases}
\end{equation}
whose solutions are in one-to-one correspondence to solutions of the conductivity equation \eqref{eq: conductivity equation} via the \emph{Liouville transformation}
\begin{equation}
\label{eq:Liouville-reduction}
    u\text{ solves \eqref{eq: conductivity equation}}\quad \Leftrightarrow\quad v=\gamma^{1/2}u\text{ solves \eqref{eq:schroedinger-equation} with }q=\frac{\Delta \gamma^{1/2}}{\gamma^{1/2}},\,\psi=\gamma^{1/2}\varphi.
\end{equation}
Clearly, this correspondence requires that $\gamma$ is sufficiently regular, say $\gamma\in W^{2,\infty}(\Omega)$ (see also \cite{haberman2015uniqueness} for low regularity results). Moreover, using the boundary determination result of Kohn and Vogelius \cite{KV84} for the inverse conductivity problem \eqref{eq: conductivity equation}, one can reduce via the Liouville reduction \eqref{eq:Liouville-reduction} the determination of $\gamma$ to the recovery of $q$ from the DN map $\Lambda_q$. Using \emph{complex geometric optics (CGO) solutions}, Sylvester and Uhlmann proved in \cite{SU87} that $q\mapsto \Lambda_q$ is injective in dimensions $n\geq 3$, which leads to the unique determination of $\gamma$ from $\Lambda_\gamma$

For the fractional Schrödinger equation \eqref{eq:fractional-schroedinger}, the Dirichlet to Neumann map is given by
\begin{equation}
\label{eq:DN-fractional-schroedinger}
    \Lambda^s_{q}\psi\vcentcolon =(-\Delta)^s v_{\psi}|_{W}
\end{equation}
for some nonempty open set $W\subset \Omega_e$, where $v_{\psi}\in H^s(\R^n)$ is the unique solution of \eqref{eq:fractional-schroedinger}. The space $H^s(\R^n)$ consists of all $L^2(\R^n)$ functions such that 
\[
    \|u\|_{H^s(\R^n)}\vcentcolon = (\|u\|_{L^2(\R^n)}^2+[u]_{H^s(\R^n)}^2)^{1/2}<\infty,
\]
where
\[
    [u]_{H^s(\R^n)}\vcentcolon =[u]_{W^{s,2}(\R^n)}\vcentcolon =\left(\int_{\R^{2n}}\frac{|u(x)-u(y)|^2}{|x-y|^{n+2s}}\,dxdy\right)^{1/2}.
\]
Indeed, the spaces $H^s(\R^n)$ can be extended to all $s\in\R$, but this will not be essential for our work.

In recent years, nonlocal inverse problems have attracted considerable attention because the fractional Laplacian enjoys the following unique continuation property.

\begin{theorem}[{UCP in energy spaces $H^r(\R^n)$,\cite{GSU20}}]
\label{thm:UCP-energy-spaces}
    Let $V\subset \R^n$ be a nonempty open set, $r\in\R$, and $0<s<1$. Suppose that $u\in H^r(\R^n)$ satisfies 
    \[
        u=(-\Delta)^s u =0\quad \text{in }V,
    \]
    then $u\equiv 0$ in $\R^n$.
\end{theorem}
This unique continuation property has been extended later to the full Bessel potential scale $H^{r,p}(\R^n)$, $r\in\R$, $1\leq p<\infty$ \cite{KRZ-2023}.

Theorem~\ref{thm:UCP-energy-spaces} in turn leads to the \emph{Runge approximation property} in a suitable Banach space $X$ for nonlocal partial differential equations (PDEs) involving the fractional Laplacian and nonlocal operators of a similar kind. For instance, in the work \cite{GSU20} it is shown that the fractional Schrödinger equation enjoys the Runge approximation property in $L^2(\Omega)$:

\begin{theorem}[$L^2$ Runge approximation property]
\label{thm:Runge-approx}
    Let $\Omega\subset\R^n$ be an open bounded set, $0<s<1$, and suppose that $q\in L^\infty (\R^n)$ is a nonnegative potential. For a given nonempty open set $W\subset \Omega_e$, the Runge set
    \[
        \mathcal{R}_W =\{v_{\psi}|_{\Omega}: v_\psi \in H^s(\R^n)\text{ solves \eqref{eq:fractional-schroedinger} with }\psi\in C_c^{\infty}(W)\}
    \]
    is dense in $L^2(\Omega)$.
\end{theorem}

The UCP of the fractional Laplacian and a corresponding Runge approximation theorem often lead to stronger uniqueness results for the Calderón-type inverse problems than in their classical limit (see, e.g.,~\cite{GU2021calder,RZ-unbounded,ghosh2021non,feizmohammadi2021fractional_closed,feizmohammadi2024calder,FKU24,RZ2022LowReg,zimmermann2023inverse,KLW2022,zimmermann2024optimalrungeapproximationdamped,cekic2020calderon,fu2026calderon,GRSU18,RS17,ruland2018exponential,ruland2021single,zimmermann2024-viscous-wave} and the references therein). 

In this note we consider the \emph{fractional conductivity equation}
\begin{equation}
\begin{cases}
\label{eq: fractional conductivity equation}
    L^s_\gamma u=0 &\text{in }\Omega,\\
    u=f &\text{in }\Omega_e,
\end{cases}
\end{equation}
where \(0<s<1\), \(\Omega\subset\mathbb R^n\) is a bounded open set,
\(\Omega_e=\mathbb R^n\setminus\overline\Omega\), and the conductivity satisfies
the uniform ellipticity condition
\begin{equation}
\label{eq: uniform ellipticity}
    0<\lambda_0\leq \gamma(x)\leq \lambda_1<\infty
    \qquad\text{for a.e. }x\in\mathbb R^n .
\end{equation}
The \emph{fractional conductivity operator} is formally given by
\begin{equation}
\label{eq: fractional conductivity operator}
    L^s_\gamma u(x)
    =
    c_{n,s}\,\operatorname{p.v.}
    \int_{\mathbb R^n}
    a(x)a(y)\frac{u(x)-u(y)}{|x-y|^{n+2s}}\,dy,\quad \text{with}\quad a(x)\vcentcolon =\gamma^{1/2}(x).
\end{equation}
Here, $c_{n,s}>0$ is chosen such that
\begin{equation}
\label{eq: constant fractional Lap}
    c_{n,s}\vcentcolon =\frac{4^s \Gamma(n/2+s)}{\pi^{n/2}|\Gamma(-s)|},
\end{equation}
because then $L^s_\gamma$ coincides for $\gamma=1$ with the fractional Laplacian \eqref{eq: frac Lap}.

For merely bounded conductivities, we interpret \(L^s_\gamma\) weakly through
the bilinear form
\begin{equation}
\label{eq: fractional bilinear form}
    B^s_\gamma(u,v)
    =
    \frac{c_{n,s}}2
    \iint_{\mathbb R^n\times\mathbb R^n}
    a(x)a(y)
    \frac{(u(x)-u(y))(v(x)-v(y))}
    {|x-y|^{n+2s}}\,dxdy .
\end{equation}
This bilinear form is naturally defined over the fractional Sobolev space $H^s(\R^n)$, it is continuous on $H^s(\R^n)$, and coercive over the space of test functions
\[
    \widetilde H^s(\Omega)=\overline{C_c^{\infty}(\Omega)}^{H^s(\R^n)}.
\]
Given exterior data \(f\in H^s(\mathbb R^n)\), \(u\in H^s(\mathbb R^n)\) is said to be a (weak) solution of
\eqref{eq: fractional conductivity equation}, if
\[
    u-f\in \widetilde H^s(\Omega),
    \qquad
    B^s_\gamma(u,\varphi)=0
    \quad\text{for all }\varphi\in \widetilde H^s(\Omega).
\]
By the fractional Poincar\'e inequality \cite[Theorem~2.2]{RZ-unbounded} and the Lax--Milgram theorem, the Dirichlet problem \eqref{eq: fractional conductivity equation} is well-posed for any $f\in H^s(\R^n)$, when $\gamma\in L^{\infty}(\R^n)$ is uniformly elliptic and $\Omega\subset\R^n$ a bounded open set. The unique solution $u_f\in H^s(\R^n)$ of \eqref{eq: fractional conductivity equation} satisfies the energy estimate
\begin{equation}
\label{eq:energy-estimate-frac-cond}
    \|u_f-f\|_{H^s(\R^n)}\leq C\|B^s_\gamma (f,\cdot)\|_{(\widetilde H^s(\Omega))^*}\leq C\|f\|_{H^s(\R^n)}.
\end{equation}
The \emph{DN map} $\Lambda^s_\gamma$, corresponding to \eqref{eq: fractional conductivity equation}, is defined weakly through the bilinear form $B^s_\gamma$:
\begin{equation}
\label{eq: DN map}
    \langle \Lambda^s_\gamma [f],[g]\rangle
=
B^s_\gamma(u_f,g),
\qquad [f],[g]\in X_s,
\end{equation}
where $X_s$ is the abstract trace space
\[
    X_s=H^s(\mathbb R^n)/\widetilde H^s(\Omega).
\]
We will often write $f$ instead of $[f]$, when no confusion can arise. Indeed, \eqref{eq:energy-estimate-frac-cond} ensures that the DN map $\Lambda^s_\gamma\colon X_s\to X_s^*$ is continuous. Concerning the above properties of the fractional conductivity operator and the related DN map, we refer the interested reader to \cite[Section~8]{RZ-unbounded}.

For an open set \(W\subset\Omega_e\), the notation
\[
    \Lambda^s_{\gamma_1}f|_W=\Lambda^s_{\gamma_2}f|_W,
    \qquad f\in C_c^\infty(W),
\]
means that
\[
    \langle(\Lambda^s_{\gamma_1}-\Lambda^s_{\gamma_2})f,g\rangle=0
    \qquad
    \text{for all }f,g\in C_c^\infty(W).
\]
If the conductivity is regular and decays sufficiently fast at infinity, then it is proved in \cite{RZ2022LowReg} that the fractional conductivity problem, that is, the recovery of $\gamma$ from the knowledge of $\Lambda^s_{\gamma}f|_W$, $f\in C_c^{\infty}(W)$, can be reduced by an exterior determination result and the fractional Liouville reduction $u\mapsto \gamma^{1/2}u$ to the fractional Calderón problem \eqref{eq:fractional-schroedinger}, \eqref{eq:DN-fractional-schroedinger}. Since we are dealing in this article with $L^\infty$ conductivities this reduction procedure is not possible and we take a different route.

\subsection{Main contribution}

The main result of this work is uniqueness for the fractional Calderón problem
with uniformly elliptic $L^\infty$-conductivities, in every dimension
$n\geq1$ and for every $0<s<1$ (Theorem~\ref{thm:uniqueness-1d} and Theorem~\ref{thm: uniqueness Linfty general}). More precisely, if
$\gamma_1,\gamma_2\in L^\infty(\mathbb R^n)$ satisfy
\[
    0<\lambda_0\leq \gamma_j\leq \lambda_1<\infty,
    \qquad j=1,2,
\]
and if their local exterior DN maps agree on a nonempty open set
$W\subset\Omega_e$, then
\[
    \gamma_1=\gamma_2
    \qquad\text{a.e. in }\mathbb R^n.
\]
Thus no continuity, differentiability, or Sobolev regularity of the
conductivity is required.

As explained above, earlier low-regularity results for the fractional conductivity problem proceed
through a fractional Liouville reduction, which transforms the conductivity
equation into a fractional Schrödinger equation \cite{RZ2022LowReg}. In particular, this
reduction requires multiplier estimates for the background deviation
\[
    m_\gamma=\gamma^{1/2}-1,
\]
and therefore imposes regularity assumptions on $\gamma$. The present
argument avoids the Liouville reduction entirely and works directly with the
conductivity bilinear form \eqref{eq: fractional bilinear form}.

In particular, this resolves the $L^\infty$-uniqueness question posed in
\cite[Question~1.10]{RZ2022LowReg}. The result can be viewed as a fractional
analogue of the Astala--Päivärinta theorem  for the two-dimensional classical
Calderón problem \cite{AP06}: in both cases, the conductivity is recovered under the
optimal qualitative assumption of uniform ellipticity and essential
boundedness. In the fractional setting considered here, the corresponding
$L^\infty$-uniqueness holds in all dimensions.

The proof also strengthens the exterior determination step. We do not assume
that $\gamma$ is continuous on the measurement set. Instead, the exterior
data concentrate at Lebesgue points, which gives
\[
    \gamma_1=\gamma_2
    \qquad\text{a.e. in }W.
\]
A further contribution is unique continuation for the fractional
Laplacian in the full weighted class $L_s$
(Theorem~\ref{theorem:bounded-fractional-ucp}).
If $g\in L_s$ vanishes almost everywhere in a nonempty open
set and $(-\Delta)^s g$ vanishes there distributionally,
then $g=0$ almost everywhere in $\mathbb R^n$.
No local boundedness or Sobolev regularity is assumed.
Since $L^p(\mathbb R^n)\subset L_s$ for every
$1\leq p\leq\infty$, the theorem applies, in particular,
to every function in any of these Lebesgue spaces. Moreover, as a corollary, we recover the unique continuation property for the Bessel potential spaces $H^{r,p}(\R^n)$, $1\leq p<\infty$, $r\in\R$ (see Corollary~\ref{cor:UCP-Bessel-potential-spaces}). This is the starting point for our global argument.

Several natural questions remain open. The method uses the same exterior set
for prescribing data and making measurements, and therefore does not settle the
genuinely disjoint partial-data problem. It also does not provide quantitative
stability estimates or a global reconstruction procedure in the full $L^\infty$
class. These problems appear to require additional ideas.

\subsection{Idea of the proof}

We describe the main ingredients. The key point is to avoid the fractional
Liouville reduction and use the nonlocal conductivity form directly.

The first step is exterior determination at Lebesgue points; see, e.g., Lemma~\ref{lemma: exterior determination for n=1 s=1/2}. Let
$x_0\in W$ be a common Lebesgue point of
\[
    a_j\vcentcolon =\gamma_j^{1/2},\qquad j=1,2.
\]
We choose exterior data $\psi_N\in C_c^\infty(W)$ concentrating at $x_0$
and normalized in the $H^s$-energy. The exterior energy estimate \cite[Lemma~3.1]{RZ2022LowReg} implies
that the corresponding solutions are close to $\psi_N$ in $H^s(\mathbb R^n)$.
Consequently the DN energy is asymptotically given by
$B^s_{\gamma_j}(\psi_N,\psi_N)$. After scaling, the relevant diagonal kernel
is an $L^1$-approximate identity, and we can demonstrate that
\[
    B^s_{\gamma_j}(\psi_N,\psi_N)\to \gamma_j(x_0).
\]
Equality of the DN maps implies $\gamma_1(x_0)=\gamma_2(x_0)$. Since this
holds at almost every point of $W$, we obtain
\[
    \gamma_1=\gamma_2
    \qquad\text{a.e. in }W.
\]

The second step is a comparison of weighted solutions (Lemma~\ref{lemma: equality of weighted solutions}). Once
$a_1=a_2$ in $W$, take exterior data $f\in C_c^\infty(W)$ which vanish
near an open set $V\Subset W$. If $u_j^f$ is the solution corresponding to
$\gamma_j$, define
\[
    h=a_1u_1^f-a_2u_2^f .
\]
Then $h=0$ in $V$. Testing the equality of the DN maps against functions
supported in $V$, and using the kernel representation of the bilinear form,
one obtains
\[
    (-\Delta)^s h=0\quad\text{in }V.
\]
The $L^2$ unique continuation property for the fractional Laplacian (Theorem~\ref{thm:UCP-energy-spaces}) then gives
\[
    a_1u_1^f=a_2u_2^f
    \qquad\text{in }\mathbb R^n.
\]

The third step recovers the coefficient globally. Choose open sets
$U,V\Subset W$ with $\overline U\cap \overline V=\emptyset$, and let $f\in C_c^\infty(W)$ satisfy
\[
    f\geq c_f>0\quad\text{on }U,
    \qquad
    f=0\quad\text{in a neighborhood of }\overline V.
\]
By the preceding comparison,
\[
    w:=a_1u_1^f=a_2u_2^f
    \qquad\text{in }\mathbb R^n.
\]
Since $u_j^f=f$ in $U\subset\Omega_e$, we have
\[
    w=a_1f=a_2f
    \quad\text{in }U,
\]
and in particular $w$ is bounded from below on $U$. Testing the DN equality
on $U$ then yields, for
\[
    \delta_a=a_1-a_2,
\]
the identities
\[
    \delta_a=0,
    \qquad
    (-\Delta)^s\delta_a=0
    \quad\text{in }U.
\]
At this point the usual $H^r(\R^n)$ unique continuation property Theorem~\ref{thm:UCP-energy-spaces} or its generalization for the $H^{r,p}(\R^n)$ space \cite{KRZ-2023} are not directly applicable, since $\delta_a$ is only bounded.

The final ingredient is unique continuation in the full tail
class $L_s$ (Theorem~\ref{theorem:bounded-fractional-ucp}).
This class contains all bounded functions, so the theorem
applies to $\delta_a$ without assumptions on its behavior
at infinity.

To prove this UCP, we consider the Caffarelli--Silvestre
extension $U(x,y)=P_y^s*g(x)$ for $g\in L_s$.
Although $g$ need not be locally bounded, its vanishing on
an open set separates the nonzero boundary data from each
compact subset of that set. Kernel estimates then give
local boundedness and weighted $H^1$-regularity of $U$
up to the corresponding boundary portion, together with
zero trace. A direct computation identifies the weighted
conormal derivative with a constant multiple of
$(-\Delta)^s g$, so it also vanishes there.
Boundary unique continuation yields local vanishing of $U$,
and interior unique continuation propagates this vanishing
throughout the upper half-space. Finally,
$P_y^s*g\to g$ in $L^1_{\mathrm{loc}}(\mathbb R^n)$
gives $g=0$ almost everywhere.

Applying this UCP to $\delta_a$ gives
\[
    a_1=a_2
    \qquad\text{a.e. in }\mathbb R^n,
\]
and hence
\[
    \gamma_1=\gamma_2
    \qquad\text{a.e. in }\mathbb R^n.
\]

For clarity, we first present the proof in the model case $n=1$,
$s=1/2$ and $g\in L^\infty (\R)$. In that case the Caffarelli--Silvestre extension is simply the
harmonic extension to the upper half-plane, and the main ideas are particularly
transparent. Moreover, along the way we establish the unique continuation property of the half Laplacian for bounded functions. We then treat the general case $n\geq1$, $0<s<1$, where the
harmonic extension is replaced by the degenerate elliptic Caffarelli--Silvestre extension and $g\in L_s$.

\subsection{Notation}

In this section we introduce the notational conventions we shall use throughout this article. Let $A$ and $B$ be subsets of $\R^n$, $n\geq 1$. The closure and the interior of $A$ are denoted by $\overline{A}$ and $\operatorname{int} A$, respectively. The distance between $A$ and $B$ is 
\[
    \dist(A,B)\vcentcolon = \inf_{x\in A,\,y\in B}|x-y|
\]
and, if $A$ is bounded, then we set
\[
    \ell_A\vcentcolon = \sup_{x\in A}|x|.
\]
Moreover, we write $A\Subset B$ if the closure of $A$ is compact and contained in the interior of $B$. For any $x_0\in \R^n$ and $r>0$, $B_r(x_0)$ denotes the open ball of radius $r$ centered at $x_0$. If $x_0=0$, then we usually write $B_r$ in place of $B_r(0)$.

If $U\subset \R^n$ is a (Lebesgue) measurable subset and $1\leq p\leq \infty$, then we write $L^p(U)$ and $L^p_{\rm loc}(U)$ for the usual Lebesgue spaces. The normalized average of a function $f\in L^1_{\rm loc}(\R^n)$ over a measurable set $A\subset \R^n$ with positive Lebesgue measure is denoted by
\begin{equation}
\label{eq:average}
    \fint_{A}f\,dx\vcentcolon = \frac{1}{|A|}\int_A f\,dx.
\end{equation}
We denote by $H^r(\R^n)$, $r\in\R$, the $L^2$-based fractional Sobolev space, that is,
\[
    H^r(\R^n)\vcentcolon =\{ u\in \tempered(\R^n):\|\ifourier ((1+|\xi|^2)^{r/2}\hat{u})\|_{L^2(\R^n)}<\infty\}.
\]
For $s\geq 0$, this space coincides with the one introduced above and the norm $\|u \|_{H^s(\R^n)}$ is equivalent to  $\|\ifourier ((1+|\xi|^2)^{r/2}\hat{u})\|_{L^2(\R^n)}$. Above $\tempered (\R^n)$ is the space of tempered distributions, which is the dual of the Schwartz space $\schwartz (\R^n)$ and it is contained in the space of distributions $\distr(\R^n)=(C_c^{\infty}(\R^n))'$, and $\hat{u}=\fourier u$ is the Fourier transform of $u\in \tempered(\R^n)$. Moreover, for open sets $\Omega\subset\R^n$, we shall make use of the space
\[
    \widetilde H^r(\Omega)
    =
    \overline{C_c^\infty(\Omega)}^{H^r(\mathbb R^n)},
\]
whose elements have zero exterior values, as well as the exterior quotient space
\[
    X_s:=H^s(\mathbb R^n)/\widetilde H^s(\Omega).
\]
We often identify an exterior class $[f]\in X_s$
with a representative $f\in H^s(\mathbb R^n)$.

Additionally, for $0<s<1$, we shall use the following weighted $L^1$ space
\begin{equation}
\label{eq: weighted space}
    L_s=
    \left\{
    u\in L^1_{\mathrm{loc}}(\mathbb R^n):
    \|u\|_{L_s}\vcentcolon = \int_{\mathbb R^n}
    \frac{|u(x)|}{1+|x|^{n+2s}}\,dx<\infty
    \right\}.
\end{equation}

The fractional Laplacian is normalized as
\begin{equation}
\label{eq: singular integral frac Lap}
    (-\Delta)^s u(x)
    \vcentcolon =
    c_{n,s}\operatorname{p.v.}
    \int_{\mathbb R^n}
    \frac{u(x)-u(y)}{|x-y|^{n+2s}}\,dy,
\end{equation}
with $c_{n,s}>0$ as in \eqref{eq: constant fractional Lap}. We also use the
Fourier multiplier definition
\begin{equation}
\label{eq: multiplier frac Lap}
    (-\Delta)^s u \vcentcolon =\mathcal F^{-1}(|\xi|^{2s}\widehat u)
\end{equation}
whenever convenient. Indeed, the multiplier definition of the fractional Laplacian makes sense for any $u\in H^r(\R^n)$, $r\in\R$, and it is continuous from $H^r(\R^n)$ to $H^{r-2s}(\R^n)$ by the Mikhlin multiplier theorem. But we can also use it for functions in $L_s$ or $\schwartz_s'(\R^n)$, which is the dual space of
\[
    \schwartz_s(\R^n)\vcentcolon = \{u\in C^{\infty}(\R^n): (1+|x|^{n+2s})\partial^\alpha u\in L^{\infty}(\R^n),\,\forall \alpha\in\N^n_0\}.
\]
We recall that the topology of $\schwartz_s(\R^n)$ is induced by the family of seminorms
\[
    [u]_k\vcentcolon =\sup_{|\alpha|\leq k}\|(1+|x|^{n+2s})\partial^\alpha u\|_{L^{\infty}(\R^n)}
\]
for $k\in\N_0$. Indeed, for $\schwartz_s'(\R^n)$ and $f\in \schwartz (\R^n)$, we may define
\begin{equation}
\label{eq: fractional Lap on Ss}
    \langle (-\Delta)^s u,f\rangle \vcentcolon = \langle u,(-\Delta)^s f\rangle,
\end{equation}
where $(-\Delta)^s f$ can be computed via the singular integral definition \eqref{eq: singular integral frac Lap} or the equivalent multiplier definition \eqref{eq: multiplier frac Lap}.

The right hand side of \eqref{eq: fractional Lap on Ss} is well-defined, since $(-\Delta)^s f\in \schwartz_s(\R^n)$. Indeed, the assignment 
\[
    \schwartz(\R^n)\ni f\mapsto (-\Delta)^s f\in \schwartz_s(\R^n)
\]
is continuous. Thus, the resulting operator $(-\Delta)^s\colon \schwartz_s'(\R^n)\to \tempered(\R^n)$, defined via \eqref{eq: fractional Lap on Ss}, is continuous. In particular, for every $u\in L_s$ and $f\in \schwartz (\R^n)$, we have
\[
    \langle (-\Delta)^s u,f\rangle=\int_{\R^n}u (-\Delta)^s f\,dx.
\]

\section{The model case: \texorpdfstring{$n=1$, $s=1/2$}{model}}
\label{sec: 1d half order}

Before proving our general uniqueness result, we present first the argument in the
one-dimensional half-order case. This case contains the main ideas of the proof
but avoids the technicalities of the Caffarelli--Silvestre extension with
degenerate weights. In particular, the bounded-data unique continuation
principle follows from the classical Poisson extension and elementary unique
continuation for harmonic functions. The general case $n\geq1$, $0<s<1$,
is treated in Section~\ref{sec: higher-dimensional generalization}. 

As we have mentioned above, our argument does not use the fractional Liouville reduction \cite{RZ2022LowReg}; it only assumes uniform ellipticity and boundedness of the conductivities. Indeed, in the case $n=1$, $s=1/2$ the methods developed in \cite{RZ2022LowReg} break down, even for smooth conductivities.

Hence, throughout this section, we suppose that $\Omega\subset\mathbb R$ is an open bounded set, $W\subset\Omega_e$ is a
nonempty open set in the exterior $\Omega_e$, for $j=1,2$ we have given two conductivities
\[
    \gamma_j\in L^\infty(\mathbb R),
    \qquad
    0<\lambda_0\leq \gamma_j\leq \lambda_1<\infty,
\]
and their square roots are denoted by
\[
    a_j (x)=\gamma_j^{1/2}(x),\quad x\in\R.
\]
To simplify the notation, we introduce the following abbreviations for the fractional conductivity operator, the bilinar form, the DN map, and the trace space: 
\[
   L_\gamma \vcentcolon = L^1_\gamma,\quad  B_\gamma(\cdot,\cdot)\vcentcolon = B^{1/2}_\gamma (\cdot,\cdot),\quad \Lambda_\gamma\vcentcolon = \Lambda^{1/2}_\gamma,\quad X\vcentcolon = X_{1/2}.
\]

\subsection{Exterior determination of bounded conductivities}

Our first result is exterior determination of bounded conductivities,
which relaxes the continuity assumption in
\cite[Theorem~3.2]{RZ2022LowReg}. Since our coefficients need not be
continuous, we first specify how pointwise assertions are understood.

In Lemma~\ref{lemma: exterior determination for n=1 s=1/2} and throughout the manuscript when we make assertions about the pointwise behaviour of a function $h\in L^1_{\rm loc}(\R^n)$, then they have always to be understood as properties of their \emph{precise representative} $h^\star\colon\R^n\to \R$, defined by
\begin{equation}
\label{eq:piontwise-representative}
    h^\star (x)\vcentcolon = \begin{cases}
        \lim_{r\downarrow  0}\fint_{B_r(x)}h\,dy,\quad&\text{when the limit exists},\\
        0,\quad&\text{otherwise}.
    \end{cases}
\end{equation}
Here we use the notation $\fint_A=|A|^{-1}\int_A$ for every Lebesgue measurable set $A\subset \R^n$ with positive finite measure $0<|A|<\infty$. This definition of $h^\star$ depends only on the equivalence class of $h$.

We call $x\in\mathbb R^n$ a \emph{Lebesgue point} of $h$ if
\begin{equation}
\label{eq:Lebesgue-point}
    \lim_{r\downarrow0}
    \fint_{B_r(x)}|h(y)-h^\star(x)|\,dy=0.
\end{equation}
By Lebesgue's differentiation theorem
\cite[Section~1.7]{evans-mt}, almost every point is a Lebesgue
point of $h$, and $h^\star=h$ almost everywhere.

In particular, if $a\in L^\infty(\mathbb R^n)$ and $\gamma=a^2$,
then every Lebesgue point $x$ of $a$ is a Lebesgue point of
$\gamma$, and
\[
    \gamma^\star(x)=(a^\star(x))^2.
\]
Indeed,
\[
    \fint_{B_r(x)}
    |a(y)^2-(a^\star(x))^2|\,dy
    \le
    2\|a\|_{L^\infty(\mathbb R^n)}
    \fint_{B_r(x)}|a(y)-a^\star(x)|\,dy
    \to 0.
\]
We suppress the superscript $\star$ in what follows.

\begin{lemma}[Exterior determination at Lebesgue points]
\label{lemma: exterior determination for n=1 s=1/2}
 Let $x_0\in W$ be a common Lebesgue point of
$a_1$ and $a_2$. If
\begin{equation}
    \label{eq:equal-DN-ext-1d}
    \Lambda_{\gamma_1}f|_W=\Lambda_{\gamma_2}f|_W
    \qquad\text{for all }f\in C_c^\infty(W),
\end{equation}
then
\[
    \gamma_1(x_0)=\gamma_2(x_0).
\]
Consequently, $\gamma_1=\gamma_2$ a.e. in $W$.
\end{lemma}

\begin{proof}
First, let us choose an open bounded interval $W_0\Subset W$ such that
\[
    x_0\in W_0\quad\text{and}\quad \dist(W_0,\Omega)>0.
\]
Choose $\psi\in C_c^\infty((-1,1))$, not identically zero, and, for $N>0$, set
\begin{equation}
\label{eq:def_psiN}
    \psi_N(x)=A\psi(N(x-x_0)),
\end{equation}
where $A>0$ is selected so that
\begin{equation}
\label{eq: normalization}
    \frac{c_{1,1/2}}{2}\int_{\mathbb R^2}
    \frac{|\psi_N(x)-\psi_N(y)|^2}{|x-y|^2}\,dx\,dy=1
\end{equation}
for every $N >0$. This is possible because the left hand side is independent of
$N$, by the change of variables $X=N(x-x_0)$, $Y=N(y-x_0)$. Moreover,
\begin{equation}
\label{eq:L2-vanishing-psiN}
    \|\psi_N\|_{L^2(\mathbb R)}^2
    =
    N^{-1}A^2\|\psi\|_{L^2(\mathbb R)}^2
    \to 0.
\end{equation}
Since $\psi\in C_c^\infty((-1,1))$ and $W_0$ is an open interval containing $x_0$, we have $\psi_N\in C_c^\infty(W_0)$ for all sufficiently large $N>0$.

Let $u_{j,N}\in H^{1/2}(\R)$ be the unique solution of $L_{\gamma_j} u=0$ in $\Omega$ with exterior value $f=\psi_N$. The exterior
energy estimate \eqref{eq:energy-estimate-frac-cond} and $\operatorname{dist}(W_0,\Omega)>0$, together with the support conditions $\supp(u_{j,N}-\psi_N)\subset \overline\Omega$ and $\supp \psi_N\subset W_0$, imply that
\begin{equation}
\label{eq:energy-estimate-ext-1d}
    \|u_{j,N}-\psi_N\|_{H^{1/2}(\mathbb R)}
    \leq C\|\psi_N\|_{L^2(W_0)}
    \to 0;
\end{equation}
see \cite[Lemma~3.1]{RZ2022LowReg}. Thus, the definition of the DN map $\Lambda_\gamma$ \eqref{eq: DN map}, $u_{j,N}-\psi_N\in \widetilde H^{1/2}(\Omega)$, and $L_{\gamma_j}u_{j,N}=0$ in $\Omega$ yield
\begin{equation}
\label{eq:DN-decomp}
    \begin{split}
    &\langle\Lambda_{\gamma_j}\psi_N,\psi_N\rangle
    =
    B_{\gamma_j}(u_{j,N},u_{j,N})\\
    &\quad = B_{\gamma_j}(\psi_N,\psi_N)+2B_{\gamma_j}(u_{j,N}-\psi_N,\psi_N)+B_{\gamma_j}(u_{j,N}-\psi_N,u_{j,N}-\psi_N)\\
    &\quad = B_{\gamma_j}(\psi_N,\psi_N)+B_{\gamma_j}(u_{j,N}-\psi_N,\psi_N)\\
    &\quad =
    B_{\gamma_j}(\psi_N,\psi_N)+o(1)
    \end{split}
\end{equation}
as $N\to\infty$. In the last step, we used the continuity of the bilinear form $B_\gamma$, \eqref{eq:energy-estimate-ext-1d}, \eqref{eq: normalization}, and \eqref{eq:L2-vanishing-psiN}.

We claim that
\begin{equation}
\label{eq:convergence-psiN-energy}
    B_{\gamma_j}(\psi_N,\psi_N)\to\gamma_j(x_0)\quad \text{as }N\to\infty.
\end{equation}
For the time bing, let us fix $j=1,2$, and set $a\vcentcolon =a_j$, $a_0\vcentcolon =a(x_0)$. Since $a\in L^{\infty}(\R)$ its precise representative is bounded,
\begin{equation}
\label{eq:product-estimate}
    |a(x)a(y)-a_0^2|
    \leq
    C\bigl(|a(x)-a_0|+|a(y)-a_0|\bigr).
\end{equation}
Using \eqref{eq: normalization} and symmetry in $x,y$, we get
\begin{equation}
\label{eq: estimate bilinear form}
\begin{aligned}
    \big|B_{\gamma_j}(\psi_N,\psi_N)-a_0^2\big|
    &\leq
    C\int_{\mathbb R}
    |a(z)-a_0|K_N(z)\,dz,
\end{aligned}
\end{equation}
where
\[
    K_N(z)
    :=
    \int_{\mathbb R}
    \frac{|\psi_N(x)-\psi_N(z)|^2}{|x-z|^2}\,dx .
\]
By scaling and \eqref{eq:def_psiN},
\begin{equation}
\label{eq: def KN}
    K_N(z)=N K(N(z-x_0))\quad\text{with}\quad
    K(Z):=A^2\int_{\mathbb R}
    \frac{|\psi(X)-\psi(Z)|^2}{|X-Z|^2}\,dX .
\end{equation}
The kernel $K$ belongs to $L^1(\mathbb R)$, because $C_c^{\infty}(\R)\subset H^{1/2}(\R)$.
By the normalization of $\psi_N$, \eqref{eq: normalization}, we have
\[
    m_K \vcentcolon =\int_{\mathbb R}K(Z)\,dZ=\frac{2}{c_{1,1/2}}.
\]
In particular
$(K_N)_{N>0}$ is an $L^1$-approximate
identity with total mass $m_K$. More explicitly, for every $\delta>0$, by Lebesgue's dominated convergence theorem, we deduce that
\begin{equation}
\label{eq: vanishing at infinity}
    \int_{|z-x_0|>\delta}K_N(z)\,dz
    =
    \int_{|Z|>N\delta}K(Z)\,dZ
    \to 0.
\end{equation}
Since $x_0$ is a Lebesgue point of $a$, we obtain the convergence
\begin{equation}
\label{eq: Lebesgue point argument}
    \int_{\mathbb R}|a(z)-a_0|K_N(z)\,dz\to 0\quad \text{as}\quad N\to\infty.
\end{equation}
Next we demonstrate this in detail. First, \eqref{eq: def KN} and the change of variables $Z=N(z-x_0)$ yield
\[
\int_{\mathbb R}|a(z)-a_0|K_N(z)\,dz
=
\int_{\mathbb R}|a(x_0+Z/N)-a_0|K(Z)\,dZ.
\]
Now we split the integral into a near part and a tail, $\int_{\mathbb R}
=
\int_{|Z|\le R}
+
\int_{|Z|>R}$, for some fixed $R>0$. For the tail, since $a$ is bounded, we get
\[
    \int_{|Z|>R}|a(x_0+Z/N)-a_0|K(Z)\,dZ
\le
2\|a\|_{L^\infty(\R)}\int_{|Z|>R}K(Z)\,dZ.
\]
Because $K\in L^1(\R)$, for any $\eps>0$, 
\begin{equation}
\label{eq: model tail estimate}
    \int_{|Z|>R}|a(x_0+Z/N)-a_0|K(Z)\,dZ <\eps
\end{equation}
for sufficiently large $R>0$. Next, we show that for every fixed $R>0$, we have
\begin{equation}
\label{eq:convergence-bulk-part}
    \int_{|Z|\le R}|a(x_0+Z/N)-a_0|K(Z)\,dZ \to 0
\end{equation}
as $N\to\infty$. Since $K\in L^1([-R,R])$ is nonnegative, there exists a step function\footnote{For any set $A\subset\R^n$, $\chi_A$ denotes the characteristic function of $A$.}
\begin{equation}
\label{eq:step-fct}
    S(Z)=\sum_{m=1}^M c_m \chi_{I_m}(Z)
\end{equation}
such that
\begin{equation}
\label{eq: approximation by step function}
    \|K\chi_{[-R,R]}-S\|_{L^1} < \eps,
\end{equation}
where, without loss of generality, we may assume that each $I_m$ is a bounded interval contained in $[-R,R]$. 

Indeed, by the standard mollification argument, there exists a nonnegative function $H\in C([-R,R])$ such that
\[
    \|K\chi_{[-R,R]}-H\|_{L^1([-R,R])}<\eps/2.
\]
Using the uniform continuity of $H$ on $[-R,R]$, we may choose a partition $-R=t_0<t_1<\cdots<t_M=R$ of $[-R,R]$ such that
\[
    |H(x)-H(y)| < \frac{\varepsilon}{4R},
\]
whenever $x,y\in [t_{m-1},t_{m}]$ for some $m=1,\ldots,M$. Now, define the step function
\[
    S(x)\vcentcolon =\sum_{m=1}^M H(\xi_m)\chi_{[t_{m-1},t_m)}(x),
\]
where $\xi_m\in [t_{m-1},t_m)$, $m=1,\ldots,M$, which has by construction the form \eqref{eq:step-fct} with
\[
    c_m\vcentcolon = H(\xi_m)\quad\text{and}\quad I_m\vcentcolon = [t_{m-1},t_m),\quad\text{for }m=1,\ldots,M.
\]
Then
\[
    \|H-S\|_{L^1([-R,R])}
\le
\sum_{m=1}^M\int_{t_{m-1}}^{t_m} |H(\xi_m)-H(x)|\,dx
\leq 
\varepsilon/2.
\]
Therefore we deduce that the step function $S$ satisfies
\[
    \|K\chi_{[-R,R]}-S\|_{L^1([-R,R])}
\le
\|K\chi_{[-R,R]}-H\|_{L^1([-R,R])}+\|H-S\|_{L^1([-R,R])}
<
\varepsilon,
\]
which proves the approximation property \eqref{eq: approximation by step function}.

Therefore, we may now write 
\begin{equation}
\label{eq: near part}
    \int_{|Z|\le R}|a(x_0+Z/N)-a_0|K(Z)\,dZ=\sum_{m=1}^M c_m \int_{I_m}|a(x_0+Z/N)-a_0|\,dZ+\mathcal{O}(\eps).
\end{equation}
We show now that each integral in \eqref{eq: near part} goes to zero as $N\to\infty$. Fix $m=1,\ldots,M$ and apply the change of variables $z=x_0+Z/N$, so that
\[
    \int_{I_m} |a(x_0+Z/N)-a_0|\,dZ
=
N\int_{x_0+I_m/N}|a(z)-a_0|\,dz.
\]
The interval $x_0+I_m/N$ has length $|I_m|/N$ and lies inside $[x_0-R/N, x_0+R/N]$.
Therefore,
\[
\begin{split}
    N\int_{x_0+I_m/N}|a(z)-a_0|\,dz
&\le
N\int_{x_0-R/N}^{x_0+R/N}|a(z)-a_0|\,dz\\
&=2R \cdot
\frac{1}{2R/N}
\int_{x_0-R/N}^{x_0+R/N}|a(z)-a_0|\,dz.
\end{split}
\]
Since $x_0\in W_0$ is a Lebesgue point of $a$, we have
\[
    \frac{1}{2R/N}
\int_{x_0-R/N}^{x_0+R/N}|a(z)-a_0|\,dz\to 0\quad \text{as }N\to\infty.
\]
Since there are finitely many pieces in \eqref{eq: near part}, we get
\[
    \limsup_{N\to\infty}\int_{|Z|\le R}|a(x_0+Z/N)-a_0|K(Z)\,dZ=\mathcal{O}(\eps).
\]
Therefore \eqref{eq: model tail estimate} ensures that
\[
    \limsup_{N\to\infty}\int_{\mathbb R}|a(z)-a_0|K_N(z)\,dz=\mathcal{O}(\eps).
\]
Finally let $\varepsilon \to 0$ to establish \eqref{eq: Lebesgue point argument}.

Consequently, \eqref{eq: estimate bilinear form} yields
\[
    B_{\gamma_j}(\psi_N,\psi_N)\to a_j(x_0)^2=\gamma_j(x_0)\quad\text{as }N\to\infty
\]
and hence proves the asserted convergence \eqref{eq:convergence-psiN-energy}.

Now, the decomposition \eqref{eq:DN-decomp} and the equality of DN maps \eqref{eq:equal-DN-ext-1d} demonstrate that
\[
    \gamma_1(x_0)=\lim_{N\to\infty}\langle\Lambda_{\gamma_1}\psi_N,\psi_N\rangle
    =
    \lim_{N\to\infty}\langle\Lambda_{\gamma_2}\psi_N,\psi_N\rangle =\gamma_2(x_0).
\]
Since almost every point of $W$ is a common Lebesgue point of $a_1,a_2$, the assertion
follows.
\end{proof}

\subsection{Equality of weighted solutions}

In this section, we show that the solutions to the two equations coincide, when they are weighted with the corresponding conductivity (see also \cite[Lemma~3.14]{RZ2022LowReg}). 

\begin{lemma}[Equality of weighted solutions]
\label{lemma: equality of weighted solutions}
Assume that $a_1=a_2$ a.e.~in $W$. Let $V\Subset W$ be a nonempty open
interval and assume that the exterior value $f\in C_c^\infty(W)$ vanishes in a neighborhood of $V$. Suppose that
\begin{equation}
\label{eq: equality of DNs weighted}
    \Lambda_{\gamma_1}f |_W=\Lambda_{\gamma_2}f|_W.
\end{equation}
For $j=1,2$, let $u_j^f\in H^{1/2}(\R)$ be the unique solution to the Dirichlet problem
\begin{equation}
\begin{cases}
\label{eq:Dirichlet-frac-cond-1d}
    L_{\gamma_j} u=0 &\text{in }\Omega,\\
    u=f &\text{in }\Omega_e,
\end{cases}
\end{equation}
then
\[
    a_1u_1^f=a_2u_2^f
    \qquad\text{a.e. in }\mathbb R .
\]
\end{lemma}

\begin{proof}
Let $h\in L^2(\R)$ be given by
\[
    h\vcentcolon =a_1u_1^f-a_2u_2^f.
\]
Since $u_j^f=f$ in $\Omega_e$, $f=0$ near $V$, and $a_1=a_2$ on
$W\supset V$, we have
\[
    h=0 \qquad\text{in }V.
\]
Next, let us recall that $(-\Delta)^{1/2} u=\ifourier( |\xi| \hat{u})$ may be defined for any $u\in L_{1/2}$ in a distributional sense and, in particular, for $h\in L^2(\R)$ (see the discussion following \eqref{eq: multiplier frac Lap}). Moreover, since $h=0$ in $V$ and $L^2(\R^n)\subset L_{1/2}$, \cite[Proposition~2.1.4]{silvestrephdthesis} demonstrates that $(-\Delta)^{1/2}$ is a continuous function in $V$ and it is given by the integral
\begin{equation}
\label{eq:half-lap-1d}
    (-\Delta)^{1/2}h(x)=c_{1,1/2}\int_{\R}\frac{h(x)-h(y)}{|x-y|^2}\,dy=- c_{1,1/2}\int_{\R\setminus V}\frac{h(y)}{|x-y|^2}\,dy
\end{equation}
for $x\in V$. Consequently, $(-\Delta)^{1/2} h\in L^1_{\text{loc}}(V)$.

Next, let $\varphi\in C_c^\infty(V)$. The equality of the DN maps, \eqref{eq: equality of DNs weighted}, gives
\begin{equation}
    B_{\gamma_1}(u_1^f,\varphi)=B_{\gamma_2}(u_2^f,\varphi).
\end{equation}
Since $u_j^f=0$ in $V$, $\supp\varphi\subset V$, and
$a_1=a_2=:a$ a.e.~in $V$, the kernel formula \eqref{eq: fractional bilinear form} reduces this identity to
\begin{equation}
\label{eq: integral identity}
    \int_V a(x)\varphi(x)
    \int_{\mathbb R\setminus V}\frac{h(y)}{|x-y|^2}\,dy\,dx=0.
\end{equation}
Indeed,
\begin{equation}
\label{eq: difference bilinear forms}
\begin{split}
    0&= \frac{2}{c_{1,1/2}}(B_{\gamma_1}(u_1^f,\varphi)-B_{\gamma_2}(u_2^f,\varphi))\\
    & =\int_{\R^2}\left\{a_1(x)a_1(y)(u_1^f(x)-u_1^f(y))-a_2(x)a_2(y)(u_2^f(x)-u_2^f(y))\right\}\\
    &\qquad \quad\times \frac{\varphi(x)-\varphi(y)}{|x-y|^2}\,dxdy.
\end{split}
\end{equation}
Now, we decompose $\R^2$ as
\[
    \R^2=(V\times V)\cup (V^c\times V^c)\cup (V\times V^c)\cup (V^c\times V).
\]
On the one hand, since $a\vcentcolon =a_1=a_2$ a.e.~in $V$, $u_1^f=u_2^f=f$ in $W$, and $f=0$ in $V$, the integrand of \eqref{eq: difference bilinear forms} vanishes in $V\times V$, and on the other hand it vanishes in $V^c\times V^c$, because $\varphi\in C_c^{\infty}(V)$. Thus, using that $a\vcentcolon = a_1=a_2$ in $V$, we get
\[
\begin{split}
    0&=\int_{V} a(x)\int_{V^c}\left(\left\{a_1(y)(u_1^f(x)-u_1^f(y))-a_2(y)(u_2^f(x)-u_2^f(y))\right\}\frac{\varphi(x)}{|x-y|^2}\,dy\right)dx\\
    &\quad -\int_{V}a(y)\int_{V^c}\left(\left\{a_1(x)(u_1^f(x)-u_1^f(y))-a_2(x)(u_2^f(x)-u_2^f(y))\right\}\frac{\varphi(y)}{|x-y|^2}\,dx\right)dy\\
    &=2\int_{V} a(x)\int_{V^c}\left(\left\{a_1(y)(u_1^f(x)-u_1^f(y))-a_2(y)(u_2^f(x)-u_2^f(y))\right\}\frac{\varphi(x)}{|x-y|^2}\,dy\right)dx.
\end{split}
\]
Since $u_j^f=f=0$ in $V$, for $x\in V$ and $y\in V^c$, we have
\[
\begin{aligned}
    &a(x)\left\{
        a_1(y)(u_1^f(x)-u_1^f(y))
        -a_2(y)(u_2^f(x)-u_2^f(y))
    \right\}\\
    &\qquad =
    -a(x)\bigl(a_1(y)u_1^f(y)-a_2(y)u_2^f(y)\bigr)\\
    &\qquad =
    -a(x)h(y).
\end{aligned}
\]
Thus, we obtain
\[
\begin{split}
    0&=\int_V a(x)\varphi(x)
    \int_{V^c}\frac{h(y)}{|x-y|^2}\,dy\,dx,
\end{split}
\]
which proves \eqref{eq: integral identity}. 

Since $(-\Delta)^{1/2} h \in L^1_{\text{loc}}(V)$, $\varphi\in C_c^{\infty}(V)$, and $0<\lambda_0^{1/2}\leq a\in L^{\infty}(\R)$, it follows from \eqref{eq:half-lap-1d} and \eqref{eq: integral identity} that
\[
    (-\Delta)^{1/2}h(x) = -c_{1,1/2}\int_{\mathbb R}\frac{h(y)}{|x-y|^2}\,dy=0
    \quad\text{a.e. in }V,
\]
and hence in $\distr(V)$. Together with $h=0$ a.e.~in $V$, this is precisely
\[
    (-\Delta)^{1/2}h=0\quad\text{in }V,
    \qquad
    h=0\quad\text{in }V.
\]
Since $h\in L^2(\mathbb R)$, the usual unique continuation property for the
half-Laplacian gives $h=0$ in $\mathbb R$; see Theorem~\ref{thm:UCP-energy-spaces}.
\end{proof}

\subsection{Half-Laplacian UCP for bounded functions}

Our next result shows that the UCP for the half-Laplacian can be extended to $L^{\infty}$-functions.

\begin{lemma}[Half-Laplacian UCP for bounded functions]
\label{lemma: bounded data UCP}
Let $I\subset\mathbb R$ be a nonempty open interval and let
$g\in L^\infty(\mathbb R)$. Suppose
\[
    g=0\quad\text{a.e. in }I,
    \qquad
    (-\Delta)^{1/2}g=0\quad\text{in }\distr(I).
\]
Then $g=0$ a.e. in $\mathbb R$.
\end{lemma}

\begin{proof}
Let
\begin{equation}
\label{eq: Poisson extension}
    G(x,y)=P_y*g(x),\qquad y>0,
\end{equation}
be the Poisson extension of $g$ to the upper half-plane, where the Poisson kernel $P_y$ is given by
\[
    P_y(x)\vcentcolon =\frac{1}{\pi}\frac{y}{x^2+y^2}
\]
for $x\in \R$, $y>0$. By \cite[Chapters~1~\&~3]{stein2016singular}, the functions $P_y$ and $G$ have the following properties:
\begin{enumerate}[(i)]
    \item\label{item: properties of P} For each $y>0$ and $1\leq p\leq \infty$, $P_y\in L^p(\R)$, and its Fourier transform is 
    \begin{equation}
    \label{eq: Fourier transform of Py}
        \widehat{P_y}(\xi)=e^{-|\xi|y},\quad y>0.
    \end{equation}
    \item\label{item: harmonic G} $P$ and $G$ are harmonic functions in $\R_+^2$.
    \item\label{item: boundedness of G} $G\in L^{\infty}(\R^2_+)$ with 
    \begin{equation}
        \|G\|_{L^{\infty}(\R^2_+)}\leq \|g\|_{L^{\infty}(\R)}.
    \end{equation}
    \item\label{item: convergence} For a.e. $x\in\R$, $G(x,y)\to g(x)$ as $y\downarrow 0$.
\end{enumerate} 
Now, let $I'\Subset I$ be a nonempty open interval. Since $g$ vanishes in $I$, the convolution $G$ is given by
\begin{equation}
\label{eq: representation of G}
    G(x,y)=\frac{1}{\pi}\int_{\R\setminus I}\frac{y}{(x-z)^2+y^2}g(z)\,dz
\end{equation}
for $x\in I'$, $y>0$. Since $|x-z|\geq \delta>0$ for all $x\in I'$ and $z\in \R\setminus I$, we infer that $G$ is smooth up to the boundary portion
$I'\times\{0\}\subset \partial \R^2_+$, that is, $G\in C^{\infty}(I'\times [0,\kappa))$, $\kappa>0$. 

To this end, let $\delta\vcentcolon=
    \operatorname{dist}(\overline{I'},\mathbb R\setminus I)>0$
and fix $\kappa>0$. For every pair of integers $k,\ell\geq0$,
there exists $C_{k,\ell,I',\delta,\kappa}>0$ such that
\[
    \left|
    \partial_x^k\partial_y^\ell
    \frac{y}{(x-z)^2+y^2}
    \right|
    \leq
    \frac{C_{k,\ell,I',\delta,\kappa}}{1+|z|^2}
\]
for $x\in\overline{I'}$, $|y|\leq\kappa$, and
$z\in\mathbb R\setminus I$.
Indeed, the identity
\[
    \frac{y}{(x-z)^2+y^2}
    =
    \frac{1}{2i}
    \left(
        \frac{1}{x-z-iy}-\frac{1}{x-z+iy}
    \right)
\]
shows that, for $k+\ell\geq1$,
\[
    \left|
    \partial_x^k\partial_y^\ell
    \frac{y}{(x-z)^2+y^2}
    \right|
    \leq
    C_{k,\ell}
    \bigl((x-z)^2+y^2\bigr)^{-(k+\ell+1)/2}.
\]
Since $|x-z|\geq\delta$, these derivatives are uniformly bounded when $z$ remains bounded. Since $I'$ is bounded, for sufficiently large $|z|$ we also have $|x-z|\geq |z|/2$, uniformly in $x\in\overline{I'}$. Thus the derivatives decay at least as $|z|^{-2}$. For $k=\ell=0$, the same decay follows from
\[
    \left|\frac{y}{(x-z)^2+y^2}\right|
    \leq \frac{\kappa}{|x-z|^2}.
\]

As $g\in L^\infty(\mathbb R)$, these bounds provide an integrable majorant for every differentiated integrand. Consequently, differentiation under the integral sign shows that
\[
    F(x,y)\vcentcolon=
    \frac{1}{\pi}\int_{\mathbb R\setminus I}
    \frac{y}{(x-z)^2+y^2}g(z)\,dz,
    \qquad (x,y)\in I'\times(-\kappa,\kappa),
\]
belongs to $C^\infty(I'\times(-\kappa,\kappa))$. For $y>0$, it agrees with $G$, and $F(x,0)=0$. Therefore $G$ is smooth up to $I'\times\{0\}$ and $G(x,0)=0$ for every $x\in I'$.

Using the same reasoning as in the proof of Lemma~\ref{lemma: equality of weighted solutions}, we conclude from the conditions $g\in L^{\infty}(\R)$ and $g=0$ in $I$ that $(-\Delta)^{1/2}g$ is continuous in $I$ and, moreover, $(-\Delta)^{1/2}g(x)$, $x\in I$, can be computed via the singular integral representation of $(-\Delta)^{1/2}$ (with $c_{1,1/2}=1/\pi$).
Additionally, $G\in C^{\infty}(I'\times [0,\kappa))$, for any $\kappa>0$, and the identity $g=(-\Delta)^{1/2}g=0$ in $I$ ensure that
\[
\begin{split}
    -\partial_yG(x,0)&=-\left.\partial_y\right|_{y=0}G(x,y)\\
    &=-\frac{1}{\pi}\left.\partial_y\right|_{y=0}\int_{\R\setminus I}\frac{y}{(x-z)^2+y^2}g(z)\,dz\\
    &=-\frac{1}{\pi}\int_{\R\setminus I}\frac{1}{(x-z)^2}g(z)\,dz\\
    &=\frac{1}{\pi}\int_{\R}\frac{g(x)-g(z)}{(x-z)^2}\,dz\\
    &=(-\Delta)^{1/2} g(x)=0
\end{split}
\]
for $x\in I'$. Thus $G$ has vanishing Cauchy data on $I'\times\{0\}$.

Since $G$ is harmonic \ref{item: harmonic G} and smooth up to $I'\times\{0\}$,
the identity $\partial_y^2G=-\partial_x^2G$ gives, for all
integers $k,m\geq0$,
\[
\begin{aligned}
    \partial_x^k\partial_y^{2m}G(x,0)
    &=(-1)^m\partial_x^{k+2m}G(x,0)=0,\\
    \partial_x^k\partial_y^{2m+1}G(x,0)
    &=(-1)^m\partial_x^{k+2m}\partial_yG(x,0)=0.
\end{aligned}
\]
Thus all derivatives of $G$ vanish on $I'\times\{0\}$.
Consequently, the zero extension
\[
    \widetilde G(x,y)\vcentcolon=
    \begin{cases}
        G(x,y),& x\in I',\ 0\leq y<\kappa,\\
        0,& x\in I',\ -\kappa<y<0
    \end{cases}
\]
belongs to $C^\infty(I'\times(-\kappa,\kappa))$ and is harmonic
there. Since this rectangle is connected and $\widetilde G$
vanishes on its lower half, real analyticity implies that
$\widetilde G=0$ throughout the rectangle.
In particular, $G$ vanishes on the nonempty open set
$I'\times(0,\kappa)$.

Applying unique continuation for harmonic functions in the
connected upper half-plane, we obtain
$G\equiv0$ in $\mathbb R^2_+$. Property~\ref{item: convergence}
then yields $g=0$ almost everywhere in $\mathbb R$.
\end{proof}

\subsection{Uniqueness of \texorpdfstring{$L^\infty$}{bounded} conductivities}

In this section, we establish the unique determination of bounded conductivities in the case $n=1$ and $s=1/2$.

\begin{theorem}[Uniqueness for $L^\infty$ conductivities when $n=1$, $s=1/2$]
\label{thm:uniqueness-1d}
Let $\Omega\subset\mathbb R$ be open bounded set, and assume that
$W\subset\Omega_e$ is a nonempty open set. Let $\gamma_1,\gamma_2\in L^\infty(\mathbb R)$ be uniformly elliptic. If
\begin{equation}
\label{eq: equality DN maps uniqueness n=1, s=1/2}
    \Lambda_{\gamma_1}f|_W=\Lambda_{\gamma_2}f|_W
    \qquad\text{for all }f\in C_c^\infty(W),
\end{equation}
then
\[
    \gamma_1=\gamma_2 \qquad\text{a.e. in }\mathbb R .
\]
\end{theorem}

\begin{proof}
By exterior determination (Lemma~\ref{lemma: exterior determination for n=1 s=1/2}), we know that $\gamma_1=\gamma_2$ a.e.~in
$W$. Next, choose nonempty open intervals $U,V\Subset W$ with $\overline U\cap\overline V=\emptyset$, and let the exterior condition
$f\in C_c^\infty(W)$ satisfy
\[
    f\geq c_f>0 \quad\text{in }U,
    \qquad
    f=0 \quad\text{in a neighborhood of }\overline V.
\]
By Lemma~\ref{lemma: equality of weighted solutions}, we have
\[
    w\vcentcolon =a_1u_1^f=a_2 u_2^f
    \qquad\text{in }\mathbb R.
\]
Thus,
\[
    a_j(x)a_j(y)(u_j^f(x)-u_j^f(y))
    =
    a_j(y)w(x)-a_j(x)w(y)
\]
for a.e.~$x,y\in\R$ and all $j=1,2$. Next, let us define the function $\delta_a\in L^\infty (\R)$ by
\[
    \delta_a(x)\vcentcolon =a_1(x)-a_2(x),\quad x\in\R.
\]
Hence, for any $\varphi\in C_c^\infty (U)$, \eqref{eq: equality DN maps uniqueness n=1, s=1/2} yields
\begin{equation}
\label{eq: difference DN maps uniqueness n=1 s=1/2}
\begin{split}
    0&=\frac{2}{c_{1,1/2}}\langle (\Lambda_{\gamma_1}-\Lambda_{\gamma_2})f,\varphi\rangle\\
    &=\frac{2}{c_{1,1/2}}(B_{\gamma_1}(u_1^f,\varphi)-B_{\gamma_2}(u_2^f,\varphi))\\
    &=\int_{\mathbb R^2}
    \frac{(\delta_a(y)w(x)-\delta_a(x)w(y))(\varphi(x)-\varphi(y))}
    {|x-y|^2}\,dx\,dy,
\end{split}
\end{equation}
Since $\gamma_1=\gamma_2$ a.e.~in $W$, we have in particular $\delta_a =0$ a.e.~in $U$. Thus $\supp\,\varphi\subset U$, the decomposition
\[
    \R^2=(U\times U)\cup (U^c\times U^c)\cup (U\times U^c)\cup (U^c\times U),
\]
and symmetry in $x$ and $y$ of the obtained integrals reduce \eqref{eq: difference DN maps uniqueness n=1 s=1/2} to
\[
    2\int_U \varphi(x)w(x)
    \int_{U^c}\frac{\delta_a (y)}{|x-y|^2}\,dy\,dx=0.
\]
Indeed, if $(x,y)\in (U\times U)\cup (U^c\times U^c)$, then the integrand of \eqref{eq: difference DN maps uniqueness n=1 s=1/2} vanishes. Hence, the symmetry in $x$ and $y$, $\delta_a=0$ in $U$, and $\varphi=0$ in $U^c$ yield
\[
    \begin{split}
        &\int_{\mathbb R^2}
    \frac{(\delta_a(y)w(x)-\delta_a(x)w(y))(\varphi(x)-\varphi(y))}
    {|x-y|^2}\,dx\,dy\\
    &=2\int_{U}\int_{U^c}
    \frac{(\delta_a(y)w(x)-\delta_a(x)w(y))(\varphi(x)-\varphi(y))}
    {|x-y|^2}\,dy\,dx\\
    &=2\int_{U}\int_{U^c}
    \frac{\delta_a(y)w(x)}
    {|x-y|^2}\varphi(x)\,dy\,dx\\
    \end{split}
\]

Now, using $w\in L^2(\R)$, $a_1=a_2$ in $W$, and
\[
    w=a_1f=a_2f\geq\lambda_0^{1/2}c_f>0\quad\text{a.e. in }U,
\]
we deduce that
\[
    \int_{U^c}\frac{\delta_a (y)}{|x-y|^2}\,dy=0
    \qquad\text{for a.e. }x\in U,
\]
and hence in $\distr (U)$. Here, we used again that $\delta_a =0$ in $U$ and $\delta_a\in L^{\infty}(\R)\subset L_{1/2}$ implies that $(-\Delta)^{1/2}\delta_a$ is a continuous function in $U$ and can be computed by
\[
    (-\Delta)^{1/2}\delta_a (x)=c_{1,1/2}\int_{\mathbb R}\frac{\delta_a(x)-\delta_a (y)}{|x-y|^2}\,dy=-c_{1,1/2}\int_{U^c}\frac{\delta_a (y)}{|x-y|^2}\,dy \in L^1_{\text{loc}}(U);
\]
see \cite[Proposition~2.1.4]{silvestrephdthesis}. Since $\delta_a =0$ a.e.~in $U$, this means
\[
    (-\Delta)^{1/2}\delta_a =0\quad\text{in }U,
    \qquad
    \delta_a =0\quad\text{in }U.
\]
The bounded-data UCP (Lemma~\ref{lemma: bounded data UCP}) implies $\delta_a =0$ a.e. in $\mathbb R$. Thus
$a_1=a_2$ a.e. in $\mathbb R$, and consequently
$\gamma_1=\gamma_2$ a.e. in $\mathbb R$.
\end{proof}

\section{The general case: \texorpdfstring{$n\geq1$, $0<s<1$}{general}}
\label{sec: higher-dimensional generalization}

In this section, we extend the model argument of
Section~\ref{sec: 1d half order} to all dimensions $n\geq1$
and all fractional orders $0<s<1$. The main new ingredient
is a unique continuation principle for $(-\Delta)^s$ with
locally bounded data satisfying the weighted integrability
condition $g\in L_s$. We establish this result in
Section~\ref{subsec:bounded-UCP-nd} using the
Caffarelli--Silvestre extension and boundary unique continuation
for the associated degenerate elliptic equation.

Section~\ref{subsec:uniqueness-nd} proves uniqueness for the inverse problem in three stages. We first establish exterior determination and then obtain equality of weighted solutions by choosing disjointly supported exterior data. Finally, we apply the unique continuation principle to the
difference of the square roots of the conductivities to conclude that $\gamma_1=\gamma_2$ almost everywhere in $\mathbb R^n$.

Throughout this section, $\Omega\subset\R^n$ is a bounded open set, $W\subset \Omega_e$ a nonempty open set, and, for $j=1,2$ we have given two conductivities
\[
    \gamma_j\in L^{\infty}(\R^n),\quad 0<\lambda_0\leq \gamma_j\leq \lambda_1<\infty,
\]
and their square roots are denoted by
\[
    a_j(x)\vcentcolon = \gamma_j^{1/2}(x),\quad x\in\R^n.
\]

\subsection{Bounded-data UCP for the fractional Laplacian}
\label{subsec:bounded-UCP-nd}

The following theorem extends
Lemma~\ref{lemma: bounded data UCP} to locally integrable functions
in arbitrary dimensions and for all $0<s<1$, which do not grow too rapidly at infinity. More explicitly, global boundedness is replaced by the weighted integrability condition
\[
    g\in L_s
    \vcentcolon=
    \left\{
        u\in L^1_{\mathrm{loc}}(\mathbb R^n):
        \|u\|_{L_s}\vcentcolon =\int_{\mathbb R^n}
        \frac{|u(x)|}{1+|x|^{n+2s}}\,dx<\infty
    \right\},
\]
which ensures that $(-\Delta)^s g$ is well-defined as a tempered
distribution by duality. In particular, since $L^p(\mathbb R^n)\subset L_s$ for every
$1\leq p\leq\infty$, the result applies to every $g\in L^p(\mathbb R^n)$, for every $1\leq p\leq\infty$.

\begin{theorem}[UCP of fractional Laplacian for $L_s$-functions]
\label{theorem:bounded-fractional-ucp}
Let $V\subset\mathbb R^n$ be a nonempty open set, and assume that $g\in L_s$. If $g$ satisfies
\[
    g=0 \quad\text{a.e.~in } V,
    \qquad
    (-\Delta)^s g=0 \quad\text{in } \distr(V),
\]
then $g=0$ a.e.~in $\mathbb R^n$.
\end{theorem}

\begin{proof}
We establish the claim in five steps.\\

\noindent\textbf{Step 1:} \textit{Caffarelli--Silvestre extension.}

We denote by
\begin{equation}
\label{eq: CS extension}
    U(x,y)\vcentcolon = P^s_y*g(x),\quad (x,y)\in\R^{n+1}_+,
\end{equation}
the Caffarelli--Silvestre (CS) extension of $g$. Here 
\[
P^s_y(x)\vcentcolon =p_{n,s}\frac{y^{2s}}{(|x|^2+y^2)^{(n+2s)/2}}
\]
is the fractional Poisson kernel and the constant $p_{n,s}>0$ is chosen such that 
\begin{equation}
\label{eq: normalization fractional poisson}
    \int_{\R^n}P^s_y(x)\,dx=1.
\end{equation}
Since $g\in L_s$, the convolution \eqref{eq: CS extension} is well-defined for every $y>0$. Indeed, for fixed $x\in\R^n$ and $y>0$, this follows from $g\in L_s\subset L^1_{\rm loc}(\R^n)$, $P_y^s\in L^\infty_{\mathrm{loc}}(\R^n)$ and the tail estimate
\begin{equation}
\label{eq: tail estimate}
    P_y^s(x-z)\leq C_{n,s}\frac{y^{2s}}{1+|z|^{n+2s}},
\end{equation}
whenever $|z|\geq R_x\vcentcolon = \max(1,2|x|)$. Indeed, this yields the pointwise estimate
\begin{equation}
\label{eq:pointwise-estiamte}
    \begin{split}
        &|P^s_y\ast g(x)|\leq \int_{\R^n}P^s_y(x-z)|g(z)|\,dz\\
        &\leq \|P^s_y(x-z)\|_{L^{\infty}(B_{R_x})}\int_{B_{R_x}}|g(z)|\,dz+C_{n,s}y^{2s}\int_{B^c_{R_x}}\frac{|g(z)|}{1+|z|^{n+2s}}\,dz\\
        &\leq \|P^s_y\|_{L^{\infty}(B_{R_x}(x))}\|g\|_{L^1(B_{R_x})}\,dz+C_{n,s}y^{2s}\|g\|_{L_s}\\
        &\leq p_{n,s}y^{-n}\|g\|_{L^1(B_{R_x})}\,dz+C_{n,s}y^{2s}\|g\|_{L_s}<\infty
    \end{split}
\end{equation}
for every $x\in\R^n$ and every $y>0$.
Hence, it remains to show \eqref{eq: tail estimate}. The condition $|z|\geq 2|x|$ yields
\[
    |x-z|\geq |z|-|x|\geq \frac{|z|}{2}.
\]
Consequently,
\[
    |x-z|^2+y^2\geq |x-z|^2\geq \frac{|z|^2}{4}.
\]
Hence
\[
\begin{split}
    P_y^s(x-z)
    &=
    p_{n,s}\frac{y^{2s}}{(|x-z|^2+y^2)^{(n+2s)/2}}  \leq
    C_{n,s}y^{2s}|z|^{-n-2s},
\end{split}
\]
where the constant $C_{n,s}>0$ only depends on $n$ and $s$. Since $|z|\geq 1$, we have
\[
    |z|^{-n-2s}
    \leq \frac{2}{1+|z|^{n+2s}}.
\]
Thus, \eqref{eq: tail estimate} follows.

Next, we show that $U\in L^{\infty}_{\rm loc}(\R^{n+1}_+)$, that is, for every compact set $K\Subset\mathbb R^n$ and all $0<r<R<\infty$,
\begin{equation}
\label{eq:uniform-bound-CS}
    \|U\|_{L^{\infty}(K\times (r,R))}\leq CR^{2s}(r^{-n-2s}\|g\|_{L^1_{\rm loc}(\R^n)}+\|g\|_{L_s})<\infty.
\end{equation}
To this end, let us define
\begin{equation}
\label{eq: radius RK}
    R_K\vcentcolon = \max(1,2 \ell_K)\quad \text{with }\ell_K\vcentcolon = \sup_{x\in K}|x|.
\end{equation}
Similar to the computation \eqref{eq:pointwise-estiamte}, we observe that \eqref{eq: tail estimate} and $0<r<R<\infty$ yield
\[
\begin{split}
    |U(x,y)|&\leq \int_{B_{R_K}}P^s_y(x-z)|g(z)|\,dz+\int_{B^c_{R_K}}P^s_y(x-z)|g(z)|\,dz\\
    &\leq C R^{2s}r^{-n-2s}\|g\|_{L^1(B_{R_K})}+CR^{2s}\|g\|_{L_s}
\end{split}
\]
for a.e.~$x\in K$ and $r<y<R$. This proves \eqref{eq:uniform-bound-CS}.

Next, we observe that $P^s_y$, $y>0$, has the following additional properties (see \cite[Lemma~7.1]{KRZ-2023}):
\begin{enumerate}[(i)]
    \item\label{item: scaling} For any $y>0$, $P^s_y(x)=y^{-n}P^s_1(x/y)$.
    \item\label{item: decreasing} For every $y>0$, $P^s_y(x)$ is decreasing in $|x|$ and $P^s_y\in L^p(\R^n)$, $1\leq p\leq \infty$.
    \item\label{item: Dirac property} For any $\delta >0$, we have
    \[
        \int_{\R^n\setminus B_\delta}P^s_y(x)\,dx\to 0\quad \text{as }y\to 0.
    \]
    \item\label{item: s harmonic} The fractional Poisson kernel is of class $C^{\infty}(\R^{n+1}_+)$ and solves the degenerate elliptic equation
    \[
        \operatorname{div}(y^{1-2s}\nabla P^s)=0\quad\text{in }\R^{n+1}_+.
    \]
\end{enumerate}
The properties \ref{item: scaling}--\ref{item: decreasing} follow by direct computation. Combining \ref{item: scaling} and
\[
    \|P^s_1\|_{L^1(\R^n)}=\|P^s_y\|_{L^1(\R^n)}=1
\]
(see \eqref{eq: normalization fractional poisson}) yields \ref{item: Dirac property}. The statement \ref{item: s harmonic} follows again by direct computation. We also note that the degenerate divergence operator can be rewritten as
\begin{equation}
\label{eq: degenerate divergence operator}
    \operatorname{div}(y^{1-2s}\nabla u)=y^{1-2s}\left(\Delta_{x,y}+\frac{1-2s}{y}\partial_y\right)u.
\end{equation}

The above properties show in particular that $(P^s_y)_{y>0}$ is an approximation of the identity in the following sense. For any
$g\in L_s$,
\begin{equation}
\label{eq: trace of U is g}
    U(\cdot,y)=P_y^s*g\to g
    \quad\text{a.e.~in }\R^n
\end{equation}
as $y\downarrow 0$. Moreover, when $g$ additionally satisfies $g\in L^{p}_{\rm loc}(\R^n)$ for some $1\leq p<\infty$, then 
\begin{equation}
\label{eq: Lp trace of U is g}
    U(\cdot,y)=P_y^s*g\to g
    \quad\text{in }L^p_{\loc}(\mathbb R^n).
\end{equation}

To see the above approximation properties of the fractional Poisson kernel, let us first note that, for any $f\in L^p(\R^n)$, $1\leq p<\infty$, we have
\begin{equation}
\label{eq: approximate identity property Lp}
    P_y^s*f\to f\quad\text{in }L^p(\R^n),
\end{equation}
as $y\downarrow 0$. This follows from the $L^p$-continuity of translations, that is
\[
    \|f(\cdot +h)-f\|_{L^p(\R^n)}\to 0
\]
as $|h| \to 0$ and the properties $P^s_y \geq 0$, \eqref{eq: normalization fractional poisson}, \ref{item: Dirac property} (see \cite[Satz~2.15]{alt1985lineare}). Next, suppose that $g\in L^p_{\rm loc}(\R^n)\cap L_s$ for some $1\leq p <\infty$. Let $r>0$. Fix radii $\rho,R>0$ such that $r<\rho<R$ and choose a cutoff function $\eta \in C_c^\infty(\mathbb R^n)$ satisfying
\[
    0\leq \eta\leq 1,\quad \left.\eta\right|_{\overline{B_\rho}}=1,\quad\text{and}\quad \left.\eta\right|_{B_{R}^c}=0.
\]
Then
\begin{equation}
\label{eq:near-tail-decomp-CS}
    P_y^s*g-g
    =
    P_y^s*(\eta g)-\eta g
    +
    P_y^s*((1-\eta)g)
    \quad\text{in }B_{\rho}.
\end{equation}
By \eqref{eq: approximate identity property Lp} and $\eta g\in L^p(\mathbb R^n)$, we have
\[
    P_y^s*(\eta g)-\eta g\to 0\quad \text{in }L^p(\R^n).
\]
Next let us show that the tail term goes to zero as $y\to 0$. Without loss of generality, we may assume that $\rho\geq \max(1,2r)$ and hence \eqref{eq: tail estimate} applies for any $x\in B_r$ and $|z|\geq \rho$. Thus, $\eta=1$ on $B_\rho$ yields
\begin{equation}
\label{eq:conv-tail-CS}
\begin{split}
    |P_y^s*((1-\eta)g)(x)|&\leq \int_{B_{\rho}^c}P^s_y(x-z)|(1-\eta(z))g(z)|\,dz\\
    &
    \leq
    C y^{2s}
    \int_{\mathbb R^n}
    \frac{|g(z)|}{1+|z|^{n+2s}}\,dz
    \to 0
\end{split}
\end{equation}
uniformly for $x\in B_r$. Thus $P_y^s*g\to g$ in $L^p(B_r)$, proving
\eqref{eq: Lp trace of U is g} whenever
$g\in L^p_{\mathrm{loc}}(\mathbb R^n)\cap L_s$. For the almost-everywhere convergence in \eqref{eq: trace of U is g}, we apply the same local/tail decomposition \eqref{eq:near-tail-decomp-CS}. More concretely, suppose that $x$ is a Lebesgue point of $g\in L_s$ and choose $r>0$ such that $x\in B_r$. Let $\eta \in C_c^{\infty}(\R^n)$ be a cutoff function as above. The tail tends to zero locally uniformly by \eqref{eq:conv-tail-CS}. Next, we show 
\begin{equation}
\label{eq:Lebesgue-point-CS-conv}
    (P^s_y\ast (\eta g)-\eta g)(x)\to 0,\quad y\downarrow 0.
\end{equation}
For convenience, set $h\vcentcolon = \eta g\in L_s \cap L^1(\R^n)$. One easily sees that $x$ is again a Lebesgue point for $h$. The function
\[
    p^s_y(t)\vcentcolon=
    p_{n,s}\frac{y^{2s}}{(t^2+y^2)^{(n+2s)/2}}.\quad t>0,
\]
satisfies
\begin{equation}
\label{eq:decay-deriv}
    0\leq -\frac{dp^s_y}{dt}(t)
    =(n+2s)p_{n,s}
    \frac{y^{2s}t}{(t^2+y^2)^{(n+2s+2)/2}}
   =o(t^{-n-s-1})\quad\text{as }t\to\infty.
\end{equation}
Consequently, \eqref{eq: normalization fractional poisson}, the polar decomposition \cite[Section~3.4.4]{evans-mt}, and integration by parts give
\[
\begin{aligned}
    |(P_y^s*h)(x)-h(x)|
    &\leq
    \int_{\mathbb R^n}|h(x-z)-h(x)|p^s_y(|z|)\,dz\\
    &=\int_{0}^\infty p^s_y(t)\left( \int_{\partial B_t}|h(x-z)-h(x)|\,d\mathcal H^{n-1}(z)\right)\,dt\\
    &=-\int_0^\infty \frac{dp^s_y}{dt}(t)\left(\int_{B_t}|h(x-z)-h(x)|\,dz\right)\,dt.
\end{aligned}
\]
From the second to third line we used \eqref{eq:decay-deriv}, 
\[
    \fint_{B_t}|h(x-z)-h(x)|\,dz\to 0
\]
because $x$ is a Lebesgue point, and the fact that for every $k\in L^1_{\rm loc}(\R^n)$ one has
\[
    \frac{d}{dt}\int_{B_t} k\,dx=\int_{\partial B_t} k\,d\mathcal{H}^{n-1}
\]
for a.e.~$t>0$. Let $\varepsilon>0$. Since $x$ is a Lebesgue point, there exists $\delta>0$ such that
\[
    \int_{B_t}|h(x-z)-h(x)|\,dz\leq\varepsilon |B_t|,
    \qquad 0<t\leq\delta.
\]
The normalization of $P_y^s$ and Tonelli's theorem imply
\[
    -\int_0^\infty\frac{dp^s_y}{dt}(t)|B_t|\,dt
    =\int_{\mathbb R^n}P_y^s(z)\,dz=1.
\]
Hence
\[
    -\int_0^\delta\frac{dp^s_y}{dt}(t)\left(\int_{B_t}|h(x-z)-h(x)|\,dz\right)\,dt
    \leq\varepsilon.
\]
For $t\geq\delta$, we use
\[
    \int_{B_t}|h(x-z)-h(x)|\,dz\leq
    \|h\|_{L^1(\mathbb R^n)}+|h(x)|\,|B_t|,
    \qquad
    -\frac{dp^s_y}{dt}(t)\leq C_{n,s}y^{2s}t^{-n-2s-1},
\]
to obtain
\[
\begin{split}
    &-\int_\delta^\infty \frac{dp^s_y}{dt}(t)\left(\int_{B_t}|h(x-z)-h(x)|\,dz\right)\,dt\\
    &\leq C_{n,s}y^{2s}
    \left(
        \|h\|_{L^1(\mathbb R^n)}\delta^{-n-2s}
        +|h(x)|\delta^{-2s}
    \right)
    \longrightarrow0
\end{split}
\]
as $y\downarrow0$. Therefore
\[
    \limsup_{y\downarrow0}|(P_y^s*h)(x)-h(x)|
    \leq\varepsilon.
\]
Letting $\varepsilon\downarrow0$ yields
$(P_y^s*h)(x)\to h(x)$, which verifies \eqref{eq:Lebesgue-point-CS-conv}.
Together with the previously established local uniform
convergence $P_y^s*((1-\eta)g)\to0$ on $B_r$, this proves
\[
    (P_y^s*g)(x)\to g(x)\quad \text{as }y\downarrow 0.
\]
Finally, since almost every point is a Lebesgue point of $g\in L^1_{\mathrm{loc}}(\mathbb R^n)$, the assertion \eqref{eq: trace of U is g} follows.

Next, we show that $U$ solves
\begin{equation}
\label{eq: CS extension PDE}
    \operatorname{div}(y^{1-2s}\nabla U)=0
    \quad\text{in }\mathbb R^{n+1}_+.
\end{equation}
By the elementary kernel estimates, see for instance
\cite[Appendix~B, estimate~(78)]{colombo2020generalized}, one has
\begin{equation}
\label{eq: estimates Poisson kernel}
    |\partial_y^k P_y^s(x)|
    \leq
    C_{n,s,k} y^{-k}P_y^s(x),
    \qquad k\in\mathbb N,
\end{equation}
and
\[
    |\partial_x^\alpha P_y^s(x)|
    \leq
    C_{\alpha,n,s}
    \frac{y^{2s}}
    {(|x|^2+y^2)^{(n+2s+|\alpha|)/2}}.
\]
In particular,
\begin{equation}
\label{eq: estimates Poisson kernel Laplacian}
    |\Delta_x P_y^s(x)|
    \leq
    C_{n,s} y^{-2}P_y^s(x).
\end{equation}
Thus, \eqref{eq: CS extension PDE} is a consequence of \ref{item: s harmonic}, \ref{item: decreasing}, \eqref{eq: degenerate divergence operator}, \eqref{eq: estimates Poisson kernel}, \eqref{eq: estimates Poisson kernel Laplacian},
and Lebesgue's dominated convergence theorem.

\medskip

\noindent\textbf{Step 2:} \textit{$H^1_{\rm loc}$-Regularity of CS extension.}

Next, let us record the local information needed at boundary points of $V\times \{0\}$. Fix
$V'\Subset V$. Since $g=0$ a.e.~in $V$, one may write
\[
    U(x,y)
    =
    \int_{\mathbb R^n\setminus V}
    P^s_y(x-z)g(z)\,dz ,\quad x\in\R^n.
\]
Let $K\Subset V'\Subset V$ and set
\begin{equation}
\label{eq: K distances}
    d_K\vcentcolon =\dist(K,\mathbb R^n\setminus V)>0.
\end{equation}
We assert that for every multi-index $\alpha\in\N^n_0$, one has
\begin{equation}
\label{eq: tangential derivatives of frac poisson}
    |\partial_x^\alpha P^s_y(x-z)|
    \leq
    C_{\alpha,K}
    \frac{y^{2s}}{(1+|x-z|^2+y^2)^{(n+2s+|\alpha|)/2}},
\end{equation}
for all $x\in K,\ y>0$, and $z\in \mathbb R^n\setminus V$.

First assume that $|\alpha|=0$. Since $K\Subset V$ implies that $1\leq d_K^{-1}|x-z|$ for $x\in K$, $z\in \R^n\setminus V$, we have
\begin{equation}
\label{eq: singularity pos distance}
    1+|x-z|^2+y^2 \leq C_K(|x-z|^2+y^2)
\end{equation}
for all $x\in K$, $y>0$, and $z\in\mathbb R^n\setminus V$. Thus, for $x\in K$ and $z\in\R^n\setminus V$, we can estimate
\[
    P^s_y(x-z)\leq C_{K}\frac{y^{2s}}{(1+|x-z|^2+y^2)^{(n+2s)/2}},
\]
which yields \eqref{eq: tangential derivatives of frac poisson} with $\alpha=0$.

Next, we consider the case $|\alpha|\geq 1$. Set
\[
    w=x-z,\quad q(w,y)=|w|^2+y^2,\quad F(t)=t^{-\beta},\quad
    \beta=\frac{n+2s}{2}.
\]
Then
\[
    P_y^s(x-z)=p_{n,s}y^{2s}F(q(w,y)).
\]
Since the $x$-derivatives are the same as the $w$-derivatives, it is enough
to estimate $\partial_w^\alpha (F\circ q)$. 

By the multivariable Fa\`a di Bruno formula,
\[
    \partial_w^\alpha(F\circ q)(w,y)
    =
    \sum_{k=1}^{|\alpha|}
    F^{(k)}(q(w,y))
    \sum_{\substack{\gamma^1+\cdots+\gamma^k=\alpha\\ |\gamma^j|\geq 1}}
    C_{\alpha,\gamma^1,\ldots,\gamma^k}
    \prod_{j=1}^k \partial_w^{\gamma^j}q(w,y).
\]
Here the expression simplifies because $q(w,y)=|w|^2+y^2$ is quadratic:
\[
    \partial_w^\gamma q=0\quad\text{if }|\gamma|\geq 3,
\]
while
\[
    |\partial_w^\gamma q|\leq C|w|
    \quad\text{if }|\gamma|=1,
    \qquad
    |\partial_w^\gamma q|\leq C
    \quad\text{if }|\gamma|=2.
\]
Moreover,
\[
    F^{(k)}(t)=C_{\beta,k}t^{-\beta-k}.
\]
Thus every nonzero term in the Fa\`a di Bruno expansion is bounded by
\[
    C q^{-\beta-k}|w|^\ell,
\]
where $\ell$ is the number of first-order derivatives of $q$ appearing in
that term. If $m$ is the number of second-order derivatives of $q$, then
\[
    k=\ell+m,
    \qquad
    |\alpha|=\ell+2m.
\]
Using $|w|^\ell\leq q(w,y)^{\ell/2}$, we obtain
\[
    q^{-\beta-k}|w|^\ell
    \leq
    q^{-\beta-k+\ell/2}
    =
    q^{-\beta-(\ell+2m)/2}
    =
    q^{-\beta-|\alpha|/2}.
\]
Consequently,
\begin{equation}
\label{eq: local bound x derivatives}
    |\partial_x^\alpha P_y^s(x-z)|
    \leq
    C_\alpha y^{2s}
    (|x-z|^2+y^2)^{-(n+2s+|\alpha|)/2}.
\end{equation}
Therefore, taking into account \eqref{eq: singularity pos distance}, we obtain
\begin{equation}
    |\partial_x^\alpha P^s_y(x-z)|
    \leq
    C_{\alpha,K}
    \frac{y^{2s}}{(1+|x-z|^2+y^2)^{(n+2s+|\alpha|)/2}} .
\end{equation}
This completes the verification of \eqref{eq: tangential derivatives of frac poisson}.

Fix $y>0$. Multiplying \eqref{eq: tangential derivatives of frac poisson} by $|g(z)|$ gives an integrable majorant of $z\mapsto \partial_x^\alpha P^s_y(x-z)g(z)$ locally uniformly for $x\in K$. Let $R_K$ be defined as in \eqref{eq: radius RK}. Then the proof of \eqref{eq: tail estimate} shows that
\begin{equation}
        \frac{1}{(1+|x-z|^2+y^2)^{(n+2s+|\alpha|)/2}} \leq \frac{C}{1+|z|^{n+2s}}
\end{equation}
for all $x\in K$, $y>0$, and $|z|\geq R_K$. Hence,
\begin{equation}
\label{eq: near far field decomposition}
\begin{split}
    &\int_{\R^n}\frac{|g(z)|}{(1+|x-z|^2+y^2)^{(n+2s+|\alpha|)/2}} \,dz\\
    &= \int_{B_{R_K}}\frac{|g(z)|}{(1+|x-z|^2+y^2)^{(n+2s+|\alpha|)/2}} \,dz+\int_{B_{R_K}^c}\frac{|g(z)|}{(1+|x-z|^2+y^2)^{(n+2s+|\alpha|)/2}} \,dz\\
    &\leq C\left(\|g\|_{L^1(B_{R_K})}+\int_{B_{R_K}^c}\frac{|g(z)|}{1+|z|^{n+2s}}\,dz\right)<\infty.
\end{split}
\end{equation}
Hence dominated convergence allows us
to pass tangential derivatives under the integral sign,
\[
    \partial_x^\alpha U(x,y)
    =
    \int_{\mathbb R^n\setminus V}
    \partial_x^\alpha P^s_y(x-z)g(z)\,dz .
\]
The same argument gives, for $x\in K$ and $y>0$,
\begin{equation}
\label{eq: local estimates}
    |U(x,y)|\leq Cy^{2s},\qquad
    |\nabla_x U(x,y)|\leq Cy^{2s},\qquad
    |\partial_y U(x,y)|\leq Cy^{2s-1}
\end{equation}
where the constant depends on $K$, on the $L^1_{\rm loc}$-norm of $g$, and
on the $L_s$-norm of $g$. For $|\alpha|\le1$, the estimate above gives
\[
    |\partial_x^\alpha U(x,y)|
    \leq
    C y^{2s}
    \int_{\mathbb R^n\setminus V}
    \frac{|g(z)|}{(1+|x-z|^2)^{(n+2s+|\alpha|)/2}}\,dz.
\]
The last integral is uniformly bounded for $x\in K$, by the estimate \eqref{eq: near far field decomposition}. This allows to obtain the first two estimates in \eqref{eq: local estimates}. 

Similarly, using \eqref{eq: estimates Poisson kernel} with $k=1$ and \eqref{eq: tangential derivatives of frac poisson} with $\alpha=0$, we have
\begin{equation}
\label{eq: uniform bound y derivative poisson}
    |\partial_yP_y^s(x-z)|\leq Cy^{-1}P^s_y(x-z)
    \leq
    C_K \frac{y^{2s-1}}{(1+|x-z|^2)^{(n+2s)/2}}
\end{equation}
for $x\in K$, $z\in \R^n\setminus V$, and $y>0$. Hence
\begin{equation}
\label{eq: uniform bound y derivative U}
    |\partial_y U(x,y)|
    \leq
    C_K y^{2s-1}
    \int_{\mathbb R^n\setminus V}
    \frac{|g(z)|}{(1+|x-z|^2)^{(n+2s)/2}}\,dz
    \leq C_K y^{2s-1},
\end{equation}
where the last integral is finite uniformly for $x\in K$. This again follows from \eqref{eq: near far field decomposition}.

As a consequence of \eqref{eq: local estimates}, we obtain
\begin{equation}
\label{eq: pointwise L2 bounds}
    y^{1-2s}|U|^2\leq Cy^{1+2s},\qquad
    y^{1-2s}|\nabla_x U|^2\leq Cy^{1+2s},\qquad
    y^{1-2s}|\partial_y U|^2\leq Cy^{2s-1},
\end{equation}
which are integrable near $y=0$ because $s>0$. Away from $y=0$, the
Poisson extension is smooth. Therefore
\begin{equation}
\label{eq: local regularity CS extension}
    U\in H^1_{\mathrm{loc}}(V'\times[0,\infty);y^{1-2s}dxdy).
\end{equation}

\medskip 

\noindent\textbf{Step 3:} \textit{Vanishing of conormal derivative.}
We claim that, for every $\varphi\in C_c^\infty(V')$,
\begin{equation}
\label{eq: vanishing conormal derivative}
    \lim_{y\downarrow 0}
    \int_{\mathbb R^n} y^{1-2s}\partial_y U(x,y)\varphi(x)\,dx
    =
    C_{n,s}\,\langle (-\Delta)^s g,\varphi\rangle
    =0,
\end{equation}
where 
\begin{equation}
\label{eq: Cs extension constant}
    C_{n,s}\vcentcolon =-2^{1-2s}\frac{\Gamma(1-s)}{\Gamma(s)}.
\end{equation}
Equivalently, $U$ has vanishing weighted conormal derivative on
$V'\times\{0\}$ in the weak sense. 

Let $K\vcentcolon = \supp\varphi$. By \eqref{eq: uniform bound y derivative U}, for $x\in K$ and $0<y<1$,
\[
    |y^{1-2s}\partial_yU(x,y)\varphi(x)|
    \leq
    C_K|\varphi(x)|.
\]
Since $\varphi\in C_c^\infty(V')$, the right hand side belongs to
$L^1(\mathbb R^n)$. Thus, once the pointwise limit of
$y^{1-2s}\partial_yU(x,y)$ is identified for $x\in K$, dominated
convergence may be applied in the $x$-integral.

For any $w\neq 0$ and $y>0$, we have
\[
\begin{split}
    y^{1-2s}\partial_y P_y^s(w)
    &=
    p_{n,s}
    \left[
        2s\frac{1}{(|w|^2+y^2)^{(n+2s)/2}}
        -(n+2s)\frac{y^2}{(|w|^2+y^2)^{(n+2s+2)/2}}
    \right].
\end{split}
\]
Thus,
\[
    y^{1-2s}\partial_y P^s_y(x-z)g(z)\to 2s p_{n,s}\frac{g(z)}{|x-z|^{n+2s}},\quad (y\to 0),
\]
for $x\in K$ and $z\in\R^n\setminus V$. Now, since $y^{1-2s}\partial_y P^s_y(x-z)g(z)$, $0<y<1$, is uniformly bounded by an integrable function (see \eqref{eq: near far field decomposition}), using dominated convergence, we may pass the $y\to 0$ limit into the $z$-integral and use the above pointwise limit to obtain
\[
\begin{split}
    \lim_{y\downarrow0}
    y^{1-2s}\partial_y U(x,y)
    &=
    2s p_{n,s}
    \int_{\mathbb R^n\setminus V}
    \frac{g(z)}{|x-z|^{n+2s}}\,dz
\end{split}
\]
for every $x\in K$. Consequently,
\[
\begin{split}
    \lim_{y\downarrow0}
    \int_{\mathbb R^n}
    y^{1-2s}\partial_y U(x,y)\varphi(x)\,dx
    &=
    2s p_{n,s}
    \int_{V'}
    \varphi(x)
    \int_{\mathbb R^n\setminus V}
    \frac{g(z)}{|x-z|^{n+2s}}\,dz\,dx .
\end{split}
\]
On the other hand, since $g=0$ in $V$, for $x\in V'$ the singular
integral defining $(-\Delta)^s g$ is nonsingular and equals
\[
    (-\Delta)^s g(x)
    =
    c_{n,s}\operatorname{p.v.}
    \int_{\mathbb R^n}
    \frac{g(x)-g(z)}{|x-z|^{n+2s}}\,dz
    =
    -c_{n,s}
    \int_{\mathbb R^n\setminus V}
    \frac{g(z)}{|x-z|^{n+2s}}\,dz .
\]
The last integral is continuous in $x\in V'$. Since, by assumption,
$(-\Delta)^s g=0$ in $\distr(V)$, this continuous representative must
vanish in $V'$. Therefore
\[
    \lim_{y\downarrow0}
    \int_{\mathbb R^n}
    y^{1-2s}\partial_y U(x,y)\varphi(x)\,dx
    =
    -\frac{2s p_{n,s}}{c_{n,s}}
    \langle (-\Delta)^s g,\varphi\rangle
    =
    0 .
\]
Finally, the value of 
\[
    C_{n,s}\vcentcolon = -\frac{2s p_{n,s}}{c_{n,s}},
\]
given in \eqref{eq: Cs extension constant}, follows from \eqref{eq: constant fractional Lap} and 
\[
    p_{n,s}=\frac{\Gamma(n/2+s)}{\pi^{n/2}\Gamma(s)}.
\]
\medskip 

\noindent\textbf{Step 4:}
\textit{Local vanishing of the CS extension.}

Fix $x_0\in V$ and choose $r>0$ such that
$\overline{B_{2r}(x_0)}\subset V$.
Set $D=B_{2r}(x_0)$ and $Q=D\times(0,2r)$. Let $K=\overline{D}$ and choose an open set $V'\subset \R^n$ such that $D\Subset V'\Subset V$. The estimates \eqref{eq: local estimates} and the regularity assertion \eqref{eq: local regularity CS extension} from Step~2, applied with $K=\overline D$ and $V'$ as above, give
\[
    U\in H^1(Q;y^{1-2s}\,dx\,dy)\cap L^\infty(Q).
\]
Moreover,
\[
    \operatorname{div}(y^{1-2s}\nabla U)=0
    \quad\text{in }Q.
\]

We first identify the weighted Sobolev trace.
For $0<t<y<2r$, the fundamental theorem of calculus and
weighted Cauchy--Schwarz imply
\[
\begin{aligned}
    \|U(\cdot,y)-U(\cdot,t)\|_{L^2(D)}
    &\leq
    \left(\int_t^y\tau^{2s-1}\,d\tau\right)^{1/2}
    \left(
        \int_t^y\int_D
        \tau^{1-2s}|\partial_\tau U|^2\,dx\,d\tau
    \right)^{1/2}\\
    &=
    \left(\frac{y^{2s}-t^{2s}}{2s}\right)^{1/2}
    \left(
        \int_t^y\int_D
        \tau^{1-2s}|\partial_\tau U|^2\,dx\,d\tau
    \right)^{1/2}.
\end{aligned}
\]
Hence $U(\cdot,y)$ has an $L^2(D)$ limit as $y\downarrow0$.
This slice limit realizes the weighted Sobolev trace. The resulting trace map is continuous in the weighted
$H^1$ norm and agrees with boundary restriction on smooth
functions. Since \eqref{eq: local estimates} also gives
\[
    \|U(\cdot,y)\|_{L^2(D)}
    \leq C|D|^{1/2}y^{2s}\to 0,
\]
the trace is zero.

We next verify the weak formulation of the zero conormal
condition. Let
$\Phi\in C_c^\infty(D\times(-2r,2r))$.
Integration by parts in $D\times(\varepsilon,2r)$ gives
\[
    \int_{D\times(\varepsilon,2r)}
    y^{1-2s}\nabla U\cdot\nabla\Phi\,dx\,dy
    =
    -\int_D
    \varepsilon^{1-2s}\partial_yU(x,\varepsilon)
    \Phi(x,\varepsilon)\,dx.
\]
Write $\Phi(x,\varepsilon)=\Phi(x,0)
+(\Phi(x,\varepsilon)-\Phi(x,0))$.
By Step~3, the boundary integral with $\Phi(x,0)$ tends
to zero. By \eqref{eq: uniform bound y derivative U},
the remaining term is bounded in absolute value by
\[
    C\|\Phi(\cdot,\varepsilon)-\Phi(\cdot,0)\|_{L^1(D)}
    \to 0.
\]
Since $U\in H^1(Q;y^{1-2s}\,dx\,dy)$ and
$y^{1-2s}$ is integrable near $y=0$, we may also pass
to the limit on the left. Thus
\[
    \int_Q y^{1-2s}\nabla U\cdot\nabla\Phi\,dx\,dy=0.
\]
This establishes the zero conormal condition in the weak
sense, including test functions that meet the flat boundary.

The restriction of $U$ to
\[
    B_{2r}^+((x_0,0))
    \vcentcolon=
    B_{2r}((x_0,0))\cap\mathbb R^{n+1}_+
\]
therefore satisfies the hypotheses of
\cite[Proposition~2.2]{ruland2015unique}.
After translation and rescaling, that proposition yields
\[
    U=0\quad\text{in }B_{2r}^+((x_0,0)).
\]
In particular,
\[
    U=0\quad\text{in }B_r(x_0)\times(0,r),
\]
since this cylinder is contained in the half-ball.
\medskip

\medskip 

\noindent\textbf{Step 5:} \textit{Global vanishing and recovery of the trace.}

It remains to pass from local vanishing of $U$ to global vanishing of $g$.
From the preceding step, there exists a nonempty open set $\mathcal O\subset \mathbb R^{n+1}_+$ such that
\[
    U=0\quad\text{in }\mathcal O.
\]
We show that this forces $U\equiv0$ in the whole upper half-space $\R^{n+1}_+$.

To this end, let
\[
    \mathcal Z
    =
    \bigl\{(x,y)\in\mathbb R^{n+1}_+:
    \text{ there exists }\rho>0\text{ such that }
    U=0\text{ in }B_\rho(x,y)\bigr\}.
\]
Then $\mathcal Z$ is nonempty and open. We claim that $\mathcal Z$ is also
closed in $\mathbb R^{n+1}_+$. Let $(x_j,y_j)\in\mathcal Z$ and suppose
\[
    (x_j,y_j)\to (x_\ast,y_\ast)\in\mathbb R^{n+1}_+.
\]
Choose $\rho>0$ such that
\[
    B_{2\rho}(x_\ast,y_\ast)\Subset \mathbb R^{n+1}_+.
\]
Since $(x_j,y_j)\in\mathcal Z$, there exists $\rho_j>0$ such that
$U=0$ in $B_{\rho_j}(x_j,y_j)$. Reducing $\rho_j$ if necessary, and using
$(x_j,y_j)\in B_\rho(x_\ast,y_\ast)$ for $j$ large enough, this gives a nonempty open subset of
$B_{2\rho}(x_\ast,y_\ast)$ on which $U$ vanishes. 

Inside $B_{2\rho}(x_\ast,y_\ast)$, the weight $y^{1-2s}$ is bounded above
and below by positive constants, and is smooth. Hence $U$ solves a uniformly
elliptic equation with smooth coefficients in this ball:
\[
    \operatorname{div}(y^{1-2s}\nabla U)=0.
\]
By the usual interior unique continuation property for uniformly elliptic
equations with (bounded) smooth coefficients, vanishing in a nonempty open subset of
$B_{2\rho}(x_\ast,y_\ast)$ implies
\[
    U=0\quad\text{in }B_{2\rho}(x_\ast,y_\ast).
\]
In particular, $(x_\ast,y_\ast)\in\mathcal Z$. Thus $\mathcal Z$ is both
open and closed in the connected set $\mathbb R^{n+1}_+$. Since
$\mathcal Z\neq\emptyset$, we conclude
\[
    U=0\quad\text{in }\mathbb R^{n+1}_+.
\]

Finally, since $g\in L_s\subset L^1_{\mathrm{loc}}(\mathbb R^n)$,
the local approximate identity property
\eqref{eq: Lp trace of U is g}, with $p=1$, gives
\[
    P_y^s*g\to g
    \quad\text{in }L^1_{\mathrm{loc}}(\mathbb R^n).
\]
Since $P_y^s*g=U(\cdot,y)=0$ for every $y>0$, we conclude
that $g=0$ almost everywhere in $\mathbb R^n$.
\end{proof}

The preceding theorem also recovers unique continuation in
Bessel potential spaces of arbitrary order, by the mollification
argument of \cite[Lemma~7.5]{KRZ-2023}.

\begin{corollary}[Unique continuation in Bessel potential spaces]
\label{cor:UCP-Bessel-potential-spaces}
Let $0<s<1$, $r\in\mathbb R$, $1\leq p<\infty$,
and let $V\subset\mathbb R^n$ be a nonempty open set.
If $u\in H^{r,p}(\mathbb R^n)$ satisfies
\[
    u=0,\qquad (-\Delta)^s u=0
    \quad\text{in }\mathcal D'(V),
\]
then $u=0$ in $\mathbb R^n$.
\end{corollary}

\begin{proof}
Let $\rho\in C_c^\infty(B_1)$ be a nonnegative, even
mollifier with integral one, and set
\[
    \rho_\varepsilon(x)=\varepsilon^{-n}\rho(x/\varepsilon),
    \qquad u_\varepsilon=\rho_\varepsilon*u.
\]
Writing $J=(1-\Delta)^{1/2}$ and $v=J^r u\in L^p$,
we have
\[
    u_\varepsilon=(J^{-r}\rho_\varepsilon)*v.
\]
Since $J^{-r}\rho_\varepsilon\in\mathcal S(\mathbb R^n)$,
Young's inequality gives $u_\varepsilon\in L^p\cap L^\infty$.
In particular, $u_\varepsilon\in L_s$.

Choose a nonempty open set $V'\Subset V$.
For $0<\varepsilon<\operatorname{dist}(\overline{V'},V^c)$,
local vanishing and commutation with mollification imply
\[
    u_\varepsilon=0,\qquad
    (-\Delta)^s u_\varepsilon
    =\rho_\varepsilon*((-\Delta)^s u)=0
    \quad\text{in }V'.
\]
Theorem~\ref{theorem:bounded-fractional-ucp} therefore gives
$u_\varepsilon=0$ in $\mathbb R^n$.
Passing to the limit as $\varepsilon\downarrow0$ in
$\mathcal D'(\mathbb R^n)$ yields $u=0$.
\end{proof}

In particular, the same unique continuation conclusion holds
for $u\in W^{r,p}(\mathbb R^n)$ whenever $r\geq0$ and
$1\leq p\leq\infty$, since
\[
    W^{r,p}(\mathbb R^n)\subset L^p(\mathbb R^n)\subset L_s.
\]

\subsection{Uniqueness proof for \texorpdfstring{$L^\infty$}{bounded} conductivities}
\label{subsec:uniqueness-nd}

In this section we adapt the proofs of Lemma~\ref{lemma: exterior determination for n=1 s=1/2}, Lemma~\ref{lemma: equality of weighted solutions}, and Theorem~\ref{thm:uniqueness-1d} to the general cases $n\geq 1$ and $0<s<1$.

\begin{lemma}[Exterior determination at Lebesgue points]
\label{lemma: exterior determination general}
Let $x_0\in W$ be a common Lebesgue point of $a_1$ and $a_2$. Assume that
\begin{equation}
    \label{eq:equal-DN-exterior-det}\Lambda^s_{\gamma_1}f|_W=\Lambda^s_{\gamma_2}f|_W
    \qquad\text{for all }f\in C_c^\infty(W).
\end{equation}
Then
\[
    \gamma_1(x_0)=\gamma_2(x_0).
\]
Consequently, $\gamma_1=\gamma_2$ a.e.~in $W$.
\end{lemma}

\begin{proof}
 Since the construction is very similar to the one of Lemma~\ref{lemma: exterior determination for n=1 s=1/2}, we limit ourselves to the main steps. Choose a bounded open set $W_0\Subset W$ such that $x_0\in W_0$ and
$\dist(W_0,\Omega)>0$. Let $\psi\in C_c^\infty(B_1(0))$, $\psi\not\equiv0$, and set
\[
    \psi_N(x)=A N^{(n-2s)/2}\psi(N(x-x_0)),\quad N>0,
\]
where $A>0$ is chosen so that
\[
    \frac{c_{n,s}}{2}
    \iint_{\mathbb R^n\times\mathbb R^n}
    \frac{|\psi_N(x)-\psi_N(y)|^2}{|x-y|^{n+2s}}\,dxdy=1.
\]
This normalization is independent of $N$ by scaling. Moreover,
\[
    \|\psi_N\|_{L^2(\mathbb R^n)}^2
    =
    A^2N^{-2s}\|\psi\|_{L^2(\mathbb R^n)}^2
    \to0.
\]
For $N>0$ sufficiently large, $\psi_N\in C_c^\infty(W_0)$.

Let $u_{j,N}\in H^s(\R^n)$ be the unique solution to the fractional conductivity equation $L_{\gamma_j}u=0$ in $\Omega$ with exterior value $\psi_N$. By the exterior
energy estimate from \cite[Lemma~3.1]{RZ2022LowReg},
\[
    \|u_{j,N}-\psi_N\|_{H^s(\mathbb R^n)}
    \leq C\|\psi_N\|_{L^2(W_0)}
    \to 0,\quad \text{as }N\to\infty.
\]
The computation \eqref{eq:DN-decomp} shows that
\[
    \langle \Lambda^s_{\gamma_j}\psi_N,\psi_N\rangle
    =
    B^s_{\gamma_j}(\psi_N,\psi_N)+o(1)\quad \text{as }N\to\infty.
\]
As in the one dimensional case, we claim that
\[
    B^s_{\gamma_j}(\psi_N,\psi_N)\to \gamma_j(x_0).
\]
Fix $j$ and write $a=a_j$, $a_0=a(x_0)$. Our starting point is again
\[
\begin{split}
    |B^s_{\gamma_j}(\psi_N,\psi_N)-a_0^2|
    &\leq
    C\int_{\mathbb R^n}|a(z)-a_0|K_N(z)\,dz,
\end{split}
\]
which follows from the estimate \eqref{eq:product-estimate}. Here we put
\[
    K_N(z):=
    \int_{\mathbb R^n}
    \frac{|\psi_N(x)-\psi_N(z)|^2}{|x-z|^{n+2s}}\,dx .
\]
By scaling, we get
\[
    K_N(z)=N^nK(N(z-x_0)),\quad \text{with}\quad K(Z)\vcentcolon =
    A^2\int_{\mathbb R^n}
    \frac{|\psi(X)-\psi(Z)|^2}{|X-Z|^{n+2s}}\,dX .
\]
Moreover $K\in L^1(\mathbb R^n)$, since $\psi\in C_c^\infty(\mathbb R^n)
\subset H^s(\mathbb R^n)$, and there holds
\[
    \int_{\mathbb R^n}K(Z)\,dZ=\frac{2}{c_{n,s}}.
\]
After changing variables, we have
\[
    \int_{\mathbb R^n}|a(z)-a_0|K_N(z)\,dz
    =
    \int_{\mathbb R^n}|a(x_0+Z/N)-a_0|K(Z)\,dZ.
\]
For fixed $R,M>0$, set $\widehat{K}_M=\min\{K,M\}$. Then
\[
\begin{aligned}
    &\int_{\mathbb R^n}|a(x_0+Z/N)-a_0|K(Z)\,dZ\\
    &\quad = \int_{B_R}|a(x_0+Z/N)-a_0|K(Z)\,dZ+\int_{\R^n\setminus B_R}|a(x_0+Z/N)-a_0|K(Z)\,dZ\\
    &\quad\leq
    M N^n\int_{B_{R/N}(x_0)}|a(z)-a_0|\,dz+2\|a\|_{L^\infty}\int_{B_R}(K-\widehat K_M)\,dZ\\
    &\qquad+
    2\|a\|_{L^\infty}
    \int_{\mathbb R^n\setminus B_R}K\,dZ.
\end{aligned}
\]
For fixed $R,M$, the first term tends to zero, since $x_0$ is a Lebesgue point of $a$ and therefore
\[
    N^n\int_{B_{R/N}(x_0)}|a(z)-a_0|\,dz
    =
    |B_R|\fint_{B_{R/N}(x_0)}|a(z)-a_0|\,dz
    \to 0,
\]
as $N\to\infty$. Taking the upper limit as $N\to\infty$, then letting
$M\to\infty$ and finally $R\to\infty$, we obtain
\[
    \int_{\mathbb R^n}|a(z)-a_0|K_N(z)\,dz\longrightarrow0,
\]
because $K\in L^1(\mathbb R^n)$.

Consequently,
\[
    B^s_{\gamma_j}(\psi_N,\psi_N)
    \to a_j(x_0)^2=\gamma_j(x_0),
    \qquad j=1,2.
\]
Combining this convergence with the DN energy decomposition
and \eqref{eq:equal-DN-exterior-det}, we obtain
\[
\begin{aligned}
    \gamma_1(x_0)
    &=\lim_{N\to\infty}
      \langle\Lambda^s_{\gamma_1}\psi_N,\psi_N\rangle\\
    &=\lim_{N\to\infty}
      \langle\Lambda^s_{\gamma_2}\psi_N,\psi_N\rangle
      =\gamma_2(x_0).
\end{aligned}
\]
Since almost every point of $W$ is a common Lebesgue point
of $a_1$ and $a_2$, it follows that
$\gamma_1=\gamma_2$ almost everywhere in $W$.
\end{proof}

\begin{lemma}[Equality of weighted solutions]
\label{lemma: equality weighted solutions general}
Assume that $a_1=a_2$ a.e.~in $W$. Let $V\Subset W$ be a nonempty open
set and assume that the exterior value $f\in C_c^\infty(W)$ vanishes in a neighborhood of $V$. Suppose that
\begin{equation}
\label{eq: equality of DNs weighted nd}
    \Lambda^s_{\gamma_1}f |_W=\Lambda^s_{\gamma_2}f|_W.
\end{equation}
For $j=1,2$, let $u_j^f\in H^{s}(\R^n)$ be the unique solution to the Dirichlet problem
\begin{equation}
\begin{cases}
\label{eq:Dirichlet-frac-cond-nd}
    L_{\gamma_j} u=0 &\text{in }\Omega,\\
    u=f &\text{in }\Omega_e,
\end{cases}
\end{equation}
then
\[
    a_1u_1^f=a_2u_2^f
    \qquad\text{a.e. in }\mathbb R^n .
\]
\end{lemma}

\begin{proof}
Again we focus on the main differences to the one-dimensional case, established in Lemma~\ref{lemma: equality of weighted solutions}. Throughout the proof, we set
\[
    h=a_1u_1^f-a_2u_2^f\in L^2(\R^n)\subset L_s.
\]
Since $u_j^f=f$ in $\Omega_e$, $f=0$ in a neighborhood of $\overline V$, and $a_1=a_2$ in $W$,
we have
\[
    h=0\quad\text{a.e. in }V.
\]
Since $h=0$ in $V$ and $h\in L_s$, \cite[Proposition~2.1.4]{silvestrephdthesis} demonstrates that $(-\Delta)^{s}$ is a continuous function in $V$ and it is given by the integral
\[
    (-\Delta)^s h(x)
    =
    -c_{n,s}\int_{\mathbb R^n \setminus V}
    \frac{h(y)}{|x-y|^{n+2s}}\,dy 
\]
for $x\in V$. Consequently, $(-\Delta)^s h\in L^1_{\rm loc}(V)$.

Let $a$ denote the common restriction of $a_1$ and $a_2$ to $V$. Since the derivation of formula \eqref{eq: integral identity} is independent of the dimension $n$ and fractional exponent $s$, we again infer from \eqref{eq: equality of DNs weighted nd} that 
\[
    \int_V a(x)\varphi(x)
    \int_{\mathbb R^n \setminus V}
    \frac{h(y)}{|x-y|^{n+2s}}\,dy\,dx=0
\]
for every $\varphi\in C_c^\infty(V)$. Since $a\geq \lambda_0^{1/2}>0$ and the inner integral belongs to $L^1_{\loc}(V)$, we obtain
\[
    (-\Delta)^s h=0\quad\text{in }\distr(V).
\]
Together with $h=0$ in $V$, the unique continuation property for the
fractional Laplacian (Theorem~\ref{thm:UCP-energy-spaces} with $r=0$) gives
\[
    h=0\quad\text{a.e. in }\mathbb R^n.
\]
Thus $a_1u_1^f=a_2u_2^f$ a.e.~in $\mathbb R^n$.
\end{proof}

\begin{theorem}[Uniqueness for $L^\infty$ conductivities]
\label{thm: uniqueness Linfty general}
Let $0<s<1$ and $n\geq 1$. Let $\Omega\subset\mathbb R^n$ be open bounded set, and assume that $W\subset\Omega_e$ is a nonempty open set. Let $\gamma_1,\gamma_2\in L^\infty(\mathbb R^n)$ be uniformly elliptic. If
\[
    \Lambda^s_{\gamma_1}f|_W=\Lambda^s_{\gamma_2}f|_W
    \qquad\text{for all }f\in C_c^\infty(W),
\]
then
\[
    \gamma_1=\gamma_2
    \qquad\text{a.e. in }\mathbb R^n.
\]
\end{theorem}

\begin{proof}
By exterior determination (Lemma~\ref{lemma: exterior determination general}), we know that $\gamma_1=\gamma_2$ a.e.~in
$W$. Next, choose nonempty open sets $U,V\Subset W$ with $\overline U\cap\overline V=\emptyset$, and let the exterior condition
$f\in C_c^\infty(W)$ satisfy
\[
    f\geq c_f>0 \quad\text{in }U,
    \qquad
    f=0 \quad\text{in a neighborhood of }\overline V.
\]
By Lemma~\ref{lemma: equality weighted solutions general}, we have
\[
    w:=a_1u_1^f=a_2u_2^f
    \qquad\text{a.e. in }\mathbb R^n.
\]
In particular, since $u_j^f=f$ in $U\subset\Omega_e$,
\[
    w=a_1f=a_2f\geq \lambda_0^{1/2}c_f>0
    \qquad\text{a.e. in }U.
\]
The rest of the proof of Theorem~\ref{thm:uniqueness-1d} is independent of the dimension $n$ or the fractional exponent $s$ and therefore we only point out the main points. Set
\[
    \delta_a:=a_1-a_2.
\]
Let $\varphi\in C_c^\infty(U)$. Since
\[
    a_j(x)a_j(y)(u_j^f(x)-u_j^f(y))
    =
    a_j(y)w(x)-a_j(x)w(y),
\]
the equality of the DN maps gives
\[
    0=
    \iint_{\mathbb R^n\times\mathbb R^n}
    \frac{(\delta_a(y)w(x)-\delta_a(x)w(y))(\varphi(x)-\varphi(y))}
    {|x-y|^{n+2s}}\,dxdy .
\]
Because $\delta_a=0$ a.e. in $U$ and $\supp\varphi\subset U$, this reduces,
by symmetry, to
\[
    \int_U \varphi(x)w(x)
    \int_{\mathbb R^n \setminus U}
    \frac{\delta_a(y)}{|x-y|^{n+2s}}\,dy\,dx=0.
\]
The inner integral belongs to $L^1_{\loc}(U)$, by \cite[Proposition~2.1.4]{silvestrephdthesis}. Because
$w\geq \lambda_0^{1/2}c_f>0$ a.e. in $U$, we get
\[
    \int_{\mathbb R^n \setminus U}
    \frac{\delta_a(y)}{|x-y|^{n+2s}}\,dy=0
    \qquad\text{for a.e. }x\in U.
\]
Equivalently,
\[
    (-\Delta)^s\delta_a=0\quad\text{in }\distr(U),
    \qquad
    \delta_a=0\quad\text{a.e. in }U.
\]
Since $\delta_a\in L^\infty(\mathbb R^n)$, Theorem~\ref{theorem:bounded-fractional-ucp}
implies
\[
    \delta_a=0\quad\text{a.e. in }\mathbb R^n.
\]
Thus $a_1=a_2$ a.e. in $\mathbb R^n$, and consequently
\[
    \gamma_1=\gamma_2\quad\text{a.e. in }\mathbb R^n.
\]
This completes the proof.
\end{proof}

	\medskip 
	
	\noindent\textbf{Acknowledgments.} 
G.~Uhlmann was partially supported by the National Science Foundation. P.~Zimmermann was supported by the Swiss National Science Foundation under Grant No.~239130.

	\bibliography{refs} 
	
	\bibliographystyle{alpha}
	
\end{document}